\documentclass[12pt,letterpaper]{amsart}
\usepackage{euscript,amsfonts,amssymb,amsmath,
amscd,amsthm,enumerate,hyperref}
\usepackage{tikz,bm,float}
\usepackage{tikz-cd}
\usetikzlibrary{calc}
\usetikzlibrary{patterns}
\usepackage{graphicx}
\usepackage{color}
\usepackage[margin=0.9in]{geometry}
\newtheorem{theorem}{Theorem} 
\newtheorem*{theorem*}{Theorem}
\newtheorem{lemma}[theorem]{Lemma}
\newtheorem{definition}[theorem]{Definition}
\newtheorem{proposition}[theorem]{Proposition}

\newtheorem{corollary}[theorem]{Corollary}

\theoremstyle{remark}
\newtheorem{remark}[theorem]{Remark}
\numberwithin{equation}{section} \numberwithin{theorem}{section}
\allowdisplaybreaks
\usepackage{MnSymbol}
\usepackage{skak}
\usepackage{adjustbox}
\usetikzlibrary{arrows}
\usetikzlibrary{decorations.pathreplacing}
\newcommand{\prech}{\prec_{\mathsf{h}}}
\newcommand{\precv}{\prec_{\mathsf{v}}}
\newcommand{\la}{\lambda}

\newcommand{\lab}{\boldsymbol\la}
\newcommand{\GT}{\mathbb{GT}^+}
\newcommand{\Z}{\mathbb{Z}}

\title[Stochastic six-vertex models and Hall-Littlewood positivity]
{Stochastic six-vertex models, Hall-Littlewood positivity and $t$-deformed Schensted insertions}

\author{Konstantin Matveev}

\address[Konstantin Matveev]{Department of Mathematics, Rutgers University, New Brunswick, NJ, USA. E-mail: kosmatveev@gmail.com}

\begin{document}

\begin{abstract} We prove a positivity theorem for a certain family of operators defined in terms of the stochastic six-vertex model. We explore connections of this result with other vertex models and $t$-deformed Schensted insertions. 
\end{abstract}

\maketitle

\section{Introduction}
\subsection{Positivity phenomena} The main result of this paper is  theorem \ref{HLpositivity}. It belongs to the following class of theorems. Suppose $P(h_{1}, h_{2}, \ldots, h_{m})$ is some polynomial with many positive and negative coefficients. Suppose each $h_i$ is itself a polynomial in $a_{1}, a_{2}, \ldots, a_{n}$. Then a priori there is no reason to expect that $P$ as a polynomial in $a_{1}, a_{2}, \ldots, a_{n}$ will have positive coefficients. If it happens to be the case, it might be an indication  that there is some structure underlying this phenomenon. Here is a relevant example. Denote by $\Lambda_{n}$ the commutative algebra of symmetric polynomials in $a_{1}, a_{2}, \ldots, a_{n}$ over $\mathbb{R}$. Define the $r$-th \emph{complete symmetric polynomial}
\begin{align}
h_{r}:= \sum_{1 \leq i_{1} \leq i_{2} \leq \cdots \leq i_{r} \leq n} a_{i_{1}}a_{i_{2}} \cdots a_{i_{r}} \in \Lambda_{n}, \qquad h_{0} = 1, \qquad h_{r} = 0 \quad \text{for } r<0.
\end{align}
\begin{proposition}
\label{BasicPositivity}
For a partition $\lambda = \{\lambda_{1} \geq \lambda_{2} \geq \cdots \geq \lambda_{\ell}\}$, $\lambda_{i} \in \mathbb{N}$, take polynomial $P$ to be $\det 
\left[ h_{\lambda_{i} - i + j} \right]_{i, j=1}^{\ell}$. Then $P$ as a polynomial in $a_{1}, a_{2}, \ldots, a_{n}$ has positive coefficients.
\end{proposition}
In this case $P$ turns out to be the Schur polynomial $S_{\lambda}(a_{1}, a_{2}, \ldots, a_{n})$ due to the Jacobi-Trudi formula. It has representation as a summation over semistandard tableaux $\lab$: 
\begin{align}
\label{BasicTableauSum}
S_{\lambda} = \sum_{\lab \text{ of shape } \lambda} a^{\lab} = \sum_{\lab \text{ of shape } \lambda} \prod_{i=1}^{\infty} a_{i}^{\lab(i)},
\end{align}
where $\lab(i)$ is the number of entries in $\lab$ equal to $i$. See \cite{F97} for more details. So coefficients of $P$ in this case are positive integers. Theorem \ref{HLpositivity} is a generalization of proposition \ref{BasicPositivity}. It comes from the following "commutative diagram" of generalizations. 

\[\begin{tikzcd}
\text{Propostion \ref{BasicPositivity}} \qquad \arrow{r}{\text{$t$-deformation}} \arrow[swap]{d}{\text{plactic lift}} & \qquad \text{Propostion \ref{HLBasicPositivity}} \arrow{d}{\text{$t$-plactic lift}} \\
\text{Propostion \ref{PlacticPositivity}} \qquad \arrow{r}{\text{$t$-deformation}} & \qquad \text{Theorem \ref{HLpositivity}} 
\end{tikzcd}
\]
Below we explain all the constituent elements of this diagram in more detail. The genesis of the paper is the following:
\begin{enumerate}
\item
Realizing that well known propositions \ref{BasicPositivity}, \ref{HLBasicPositivity} and less well known proposition \ref{PlacticPositivity} are true and a guess that theorem \ref{HLpositivity} might be true as well. 

\smallskip

\item
Checking correctness of theorem \ref{HLpositivity} for many specific partitions $\lambda$ with the assistance of \emph{Wolfram Mathematica}.

\smallskip

\item
Using the angle of $t$-deformed plactic algebra actions to find the right generalization of the summation formula \eqref{BasicTableauSum} from which theorem \ref{HLpositivity} follows. 

\smallskip

\item
Realizing that extended vertex models provide the right framework for expressing and proving such generalization.  

\end{enumerate}

\subsection{Hall-Littlewood symmetric polynomials} 
\label{HL polynomials}
Hall-Littlewood symmetric polynomial  $P_{\lambda}$ is a $t$-deformation of the Schur polynomial $S_{\lambda}$. It becomes $S_{\lambda}$ for $t=0$. See \cite{Mac99} for more details and \cite{H59}, \cite{L61} for the origin of the concept. One possible definition of $P_{\lambda}(a_{1}, a_{2}, \ldots, a_{n})$ is the following. Denote by $m_{k}(\lambda)$ the number of parts $\lambda_{i}$, which are equal to $k$. Then 
\begin{align*}
P_{\lambda}:= \frac{1}{v_{\lambda}(t)} \sum_{w \in S_{n}} w\left(a_{1}^{\lambda_{1}}a_{2}^{\lambda_{2}} \cdots a_{n}^{\lambda_{n}} \prod_{1 \leq i < j \leq n}\frac{a_{i}-ta_{j}}{a_{i}-a_{j}} \right), \quad \text{where } v_{\lambda}(t) := \prod_{k \geq 0} \prod_{i=1}^{m_{k}(\lambda)} \frac{1-t^{i}}{1-t}.
\end{align*}
\begin{proposition}
\label{HLBasicPositivity}
For $0 \leq t < 1$ all the coefficients of $P_{\lambda}$ are non-negative. 
\end{proposition}
This result follows from the following summation formula
\begin{align}
\label{HLTableauSum}
P_{\lambda} = \sum_{\lab  \text{ of shape } \lambda} \psi_{\lab}(t) a^{\lab}
\end{align}
Here tableau weight $\psi_{\lab}(t)$ can be defined as the product 
$\displaystyle \prod_{s \in \text{ Boxes of } \lab}\left(1-e(s)\right)$, where for a box $s$ with entry $p$ in the $i$-th row and the $j$-th column of $\lab$ we define
\begin{multline*}
e(s):= \begin{cases} 
      0, & \text{if $j=1$ or the $(j-1)$-st column} \\ & \text{of $T$ also contains $p$}, \\
      t^{\lvert\{\text{entries $<p$ in the $(j-1)$-st column of $T$\}}\rvert-i+1}, & \text{otherwise.}
   \end{cases}
\end{multline*}
All product terms are clearly non-negative for $0 \leq t < 1$. Proposition \ref{BasicPositivity} follows from proposition \ref{HLBasicPositivity} by substituting $t=0$.
\begin{figure}[h]
\label{HLTableaux}
\includegraphics[width = 
0.4\textwidth]{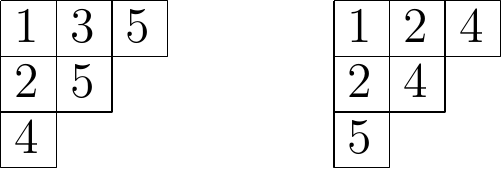}
\caption{$\lambda = (3, 2, 1)$, $n=5$. \emph{Left:} A tableau with weight $\psi(t) = \left( 1-t^{2}\right)^{2}$. \emph{Right:} A tableau with weight $\psi(t) = 1-t$.}
\end{figure}

Let $g_{k}=(1-t)P_{(k)}$. For $t=0$ polynomial $g_{k}$ becomes $h_k$. Then it is easy to show that $\Lambda_{n} = \mathbb{R}[g_{1}, g_{2}, \ldots, g_{n}]$ and factorization of the generating function for $g_{k}'s$
\begin{align}
\label{generatingfunction}
1+\sum_{k=1}^{\infty} g_{k}\alpha^{k} = \prod_{i=1}^{n} \frac{1-t\alpha a_{i}}{1 - \alpha a_{i}}
\end{align}

\subsection{Plactic algebra action}
\label{PlacticAlgebraActions}
\emph{Plactic monoid} $Pl_{n}$ of rank $n$ is a monoid generated by letters $\mathsf{1}, \mathsf{2}, \ldots, \mathsf{n}$ modulo the {\it Knuth relations}:
\begin{align*}
		xzy \equiv zxy \quad (x \leq y < z), \qquad
                     yxz \equiv yzx \quad (x < y \leq z).
\end{align*}
See \cite{Lot02} for more details. 

\begin{definition}
$Pl_{n}$ acts on the set of subsets of $\{1, 2, \ldots, n\}$ via the following formula.  For $S \subseteq \{1, 2, \ldots, n\}$ define
\begin{align}
\mathsf{i} \cdot S = 
\begin{cases}
S, \quad \text{if } i \in S, \\
S\cup\{i\}, \quad \text{if } i \notin S \text{ and } S\cap \{i+1, \ldots, n\} = \emptyset, \\
S\cup\{i\} - \left \{ \text{First element of $S\cap \{i+1, \ldots, n\}$}\right\}, \quad \text{if } i \notin S \text{ and } S\cap \{i+1, \ldots, n\} \neq \emptyset.
\end{cases}
\end{align}
Consider the plactic algebra $\mathbb{R}[a_{1}, a_{2}, \ldots, a_{n}][Pl_{n}]$ spanned by the monoid $Pl_{n}$. By linearly extending the previous action we get that $\mathbb{R}[a_{1}, a_{2}, \ldots, a_{n}][Pl_{n}]$ acts on the space of formal linear combinations of subsets of $\{1, 2, \ldots, n\}$.
\end{definition}
Consider the following commuting elements of $\mathbb{R}[a_{1}, a_{2}, \ldots, a_{n}][Pl_{n}]$.

\begin{multline}
H_{r}:= \sum_{1 \leq i_{1} \leq i_{2} \leq \cdots \leq i_{r} \leq n} (a_{i_{1}}a_{i_{2}} \cdots a_{i_{r}}) \ \mathsf{i_{1}} \cdot  \mathsf{i_{2}} \cdot  \cdots \cdot \mathsf{i_{r}} \in \mathbb{R}[a_{1}, a_{2}, \ldots, a_{n}][Pl_{n}], \\ H_{0} = Id, \qquad H_{r} = 0 \quad \text{for } r<0.
\end{multline}

\begin{proposition}
\label{PlacticPositivity}
For any partition $\lambda$ element  $\det 
\left[ H_{\lambda_{i} - i + j} \right]_{i, j=1}^{\ell} \in  \mathbb{R}[a_{1}, a_{2}, \ldots, a_{n}][Pl_{n}]$ sends any subset $S \subseteq \{1, 2, \ldots, n\}$ to a linear combination of subsets of $\{1, 2, \ldots, n\}$ with each coefficient being a polynomial in $a_{1}, a_{2}, \ldots, a_{n}$ with positive coefficients. 
\end{proposition}
Note that $\mathsf{i} \cdot =\{1, 2, \ldots, n\} = \{1, 2, \ldots, n\}$ for any $1 \leq i \leq n$. Hence $\det 
\left[ H_{\lambda_{i} - i + j} \right]_{i, j=1}^{\ell} \cdot \{1, 2, \ldots, n\} = S_{\lambda}(a_{1}, a_{2}, \ldots, a_{n}) \cdot \{1, 2, \ldots, n\}$. So proposition \ref{BasicPositivity} is a special case of proposition \ref{PlacticPositivity} for $S =\{1, 2, \ldots, n\}$. The reason behind proposition \ref{PlacticPositivity} is combinatorics of the Schensted insertion algorithm. This connection is explained in more detail in section \ref{classicRSK}.

\subsection{Stochastic six-vertex model}
We will work with inhomogeneous transfer matrices of a stochastic six-vertex model. Let $V$ be a two-dimensional real vector space spanned by elements $\mathsf{1}$ and $\mathsf{2}$. Given two parameters $a$, $t$ one can define an operator $R = R(a, t): V^{\otimes 2} \to V^{\otimes 2}$ by 
\begin{multline}
R(\mathsf{2} \otimes \mathsf{2}) = \mathsf{2} \otimes \mathsf{2}, \qquad R(\mathsf{2} \otimes \mathsf{1}) = \frac{1-a}{1-ta} \cdot \mathsf{2} \otimes \mathsf{1} + \frac{(1-t)a}{1-ta} \cdot \mathsf{1} \otimes \mathsf{2}, \\
R(\mathsf{1} \otimes \mathsf{2}) = \frac{1-t}{1-ta} \cdot \mathsf{2} \otimes \mathsf{1}+ \frac{t(1-a)}{1-ta} \cdot \mathsf{1} \otimes \mathsf{2}, \qquad R(\mathsf{1} \otimes \mathsf{1}) = \mathsf{1} \otimes \mathsf{1}.
\end{multline}
For $0 \leq a \leq 1$ and $0 \leq t < 1$ the matrix of $R$ with respect to basis $\{\mathsf{2} \otimes \mathsf{2}, \mathsf{2} \otimes \mathsf{1}, \mathsf{1} \otimes \mathsf{2}, \mathsf{1} \otimes \mathsf{1}\}$ of $V^{\otimes 2}$ is stochastic. This $R$-matrix gives rise to a stochastic six-vertex model with weights as depicted on Fig. \ref{Stochastic6V}.

\begin{figure}[h]
\includegraphics[width = 
0.9\textwidth]{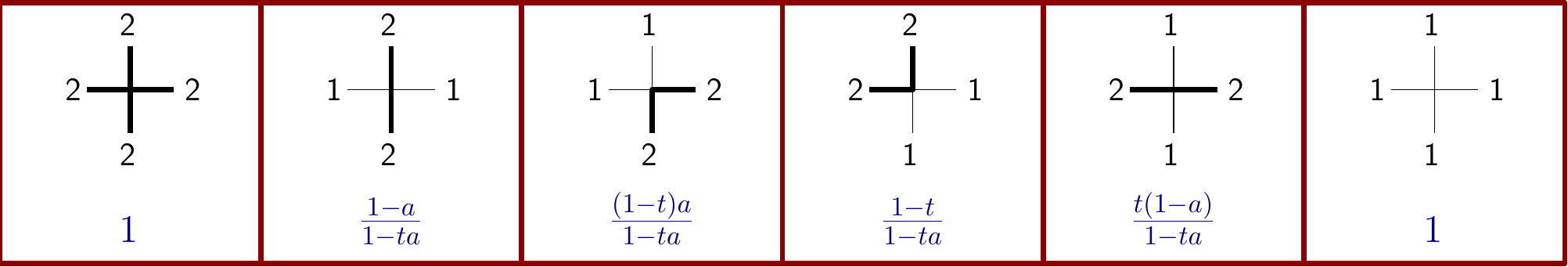}
\caption{Stochastic six-vertex model with weights coming from $R(a, t)$. Fat lines correspond to $\mathsf{2}$'s, normal lines correspond to $\mathsf{1}$'s.}
\label{Stochastic6V}
\end{figure}
For parameters $\alpha, t, a_{1}, a_{2}, \ldots, a_{n}$ we use $R$ to define a transfer operator of a six-vertex model 
\begin{align*}
T(\alpha) = T(\alpha, t; a_{1}, a_{2}, \ldots, a_{n}): V^{\otimes n} \to V^{\otimes n}
\end{align*} as specified on Fig. \ref{TransferMatrix}. Note that there is fixed input $\mathsf{1}$ on the left, while boundary condition on the right is free (it can be either $\mathsf{1}$ or $\mathsf{2}$). 
\begin{figure}[h]
\includegraphics[width = 
0.9\textwidth]{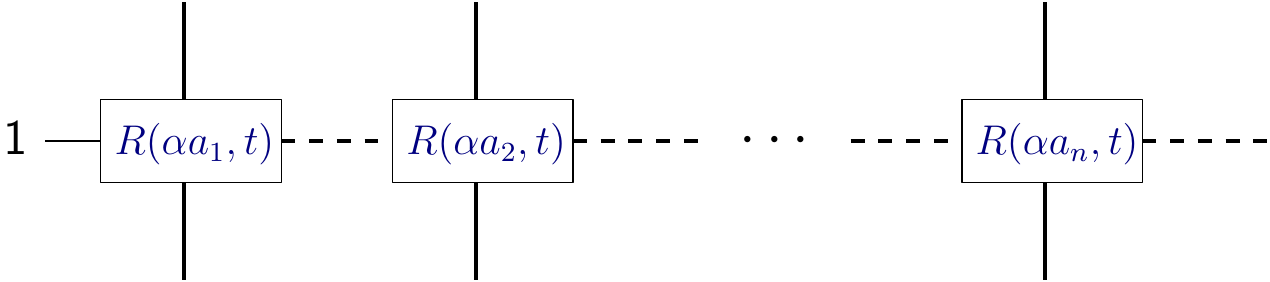}
\caption{$T(\alpha)$ is a transfer operator for the inhomogeneous six-vertex model defined via $R$. Boundary condition is fixed to be $\mathsf{1}$ on the left and is free on the right.} 
\label{TransferMatrix}
\end{figure}
For $0 \leq \alpha a_{1}, \alpha a_{2}, \ldots, \alpha a_{n} \leq 1$ and $0 \leq t < 1$ the matrix of $T(\alpha)$ with respect to basis $\{\mathsf{1}, \mathsf{2}\}^{\otimes n}$ of $V^{\otimes n}$ is stochastic. $V^{\otimes n} = \bigoplus_{k=0}^{n} V_{n, k}$, where each subspace $V_{n, k}$ is defined as the span of vectors $e_{1} \otimes e_{2} \otimes \cdots \otimes e_{n}$ with exactly $k$ $\mathsf{1}$'s and $n-k$ $\mathsf{2}$'s among the $e_{i}$'s. Then clearly $T(\alpha)(V_{n, k}) \subset V_{n, k} \oplus V_{n, k+1}$ for $0 \leq k < n$. Operators $T(\alpha)$ and $T(\beta)$ commute via a standard argument of pulling an extra vertex through due to the Yang-Baxter equation for $R$ specified on  Fig. \ref{YangBaxter}.
\begin{figure}[h]
\includegraphics[width = 
0.9\textwidth]{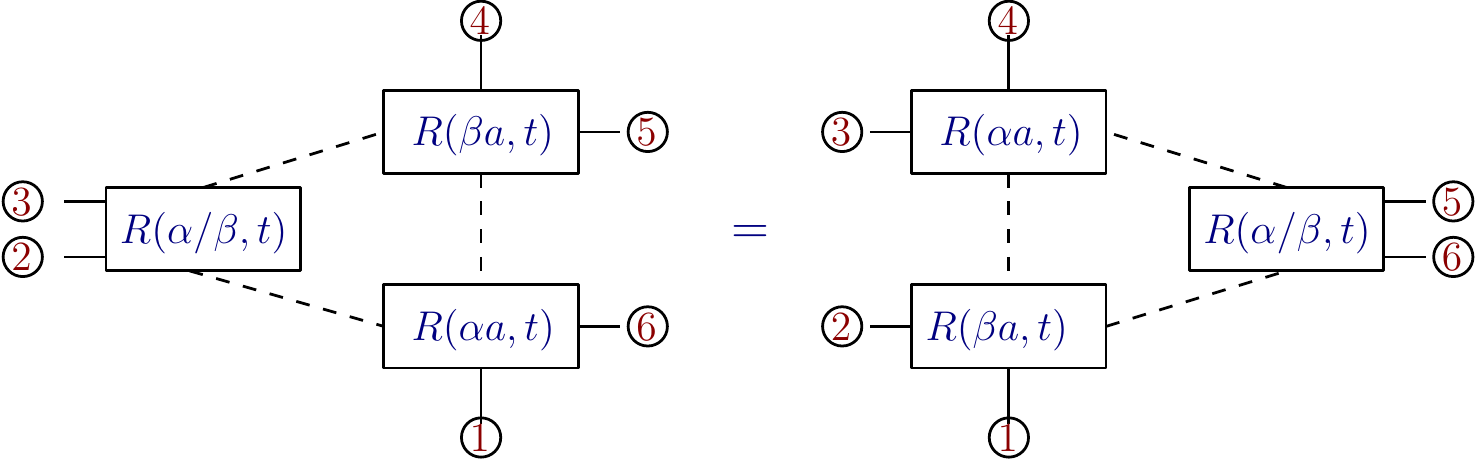}
\caption{Yang-Baxter equation for matrix $R$.}
\label{YangBaxter}
\end{figure}
Let 
\begin{align}
\label{definitionmod}
\widetilde{T}(\alpha) := \left( \prod_{i=1}^{n} \frac{1-t\alpha a_{i}}{1- \alpha a_{i}} \right) T(\alpha).
\end{align}
Then $\widetilde{T}(\alpha)$ is the transfer operator for the inhomogeneous six-vertex model with weights specified on Fig. \ref{Stochastic6Vmod}.
\begin{figure}[h]
\includegraphics[width = 
0.9\textwidth]{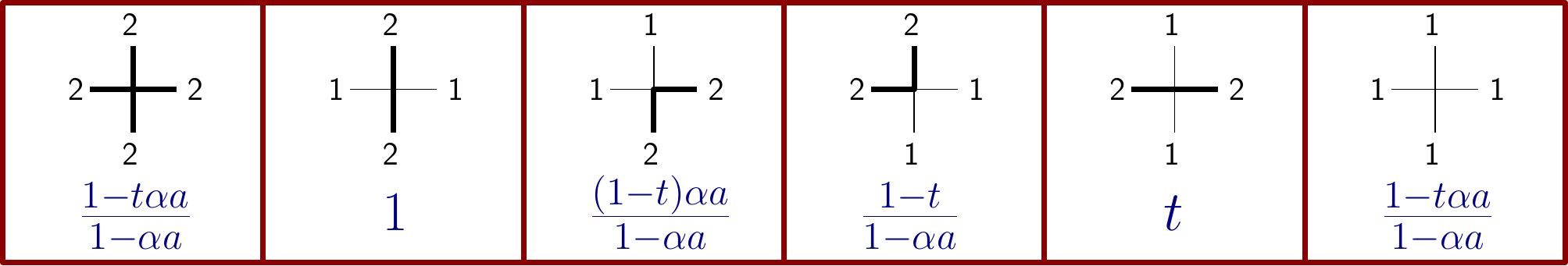}
\caption{Six-vertex model weights for $\widetilde{T}(\alpha)$.}
\label{Stochastic6Vmod}
\end{figure}
Consider the power expansion $\widetilde{T}(\alpha) = Id + \sum_{k=1}^{\infty} T_{k} \alpha^{k}$.  Operators $\widetilde{T}(\alpha)$ and $\widetilde{T}(\beta)$ commute with each other, hence operators $T_k$ and $T_{\ell}$ commute for all positive integers $k, \ell$. 
\begin{definition} Define representation $\Theta: \Lambda_{n} \to End(V^{\otimes n})$ by  $\Theta(g_{k}) = T_{k}$ for $1 \leq k \leq n$. We will later check that this equality also holds for $k >n$.  
\end{definition}
\subsection{Statement of the main result}
Our main result is the following
\begin{theorem}
\label{HLpositivity}
For $0 \leq t < 1$ and any partition $\lambda$ all matrix elements of $\Theta(P_{\lambda})$ with respect to basis $\{\mathsf{1}, \mathsf{2}\}^{\otimes n}$ are polynomials in $a_{1}, a_{2}, \ldots, a_{n}$ with non-negative coefficients.  
\end{theorem}
\begin{corollary}
For $a_{1}, a_{2}, \ldots, a_{n} > 0$ and any partition $\lambda$ with $\lambda'_{1} \leq n$ the matrix of $\frac{\Theta(P_{\lambda})}{P_{\lambda}(a_{1}, a_{2}, \ldots, a_{n})}$ with respect to basis $\{\mathsf{1}, \mathsf{2}\}^{\otimes n}$ is stochastic.
\end{corollary}

\begin{proof}[Proof of the corollary \ref{HLpositivity}] Take $ v \in \{\mathsf{1}, \mathsf{2}\}^{\otimes n}$. Then sum of basis coefficients of $T(\alpha)(v)$ is $1$, since weights of $T(\alpha)$ are stochastic. So sum of basis coefficients of $\widetilde{T}(\alpha)(v)$ is $\prod_{i=1}^{n} \frac{1-t\alpha a_{i}}{1 - \alpha a_{i}} = 1+\sum_{k=1}^{\infty} g_{k}(a_{1}, a_{2}, \ldots, a_{n})\alpha^{k}$ due to \eqref{generatingfunction}. So sum of basis coefficients of $T_{k}(v)$ is $g_{k}(a_{1}, a_{2}, \ldots, a_{n})$. Hence sum of basis coefficients of $T_{k}(v)$ is $P_{\lambda}(a_{1}, a_{2}, \ldots, a_{n})$. Conditions $\lambda'_{1} \leq n$ and $0 \leq t < 1$ guarantee that $P_{\lambda}(a_{1}, a_{2}, \ldots, a_{n}) > 0$ due to formula \eqref{HLTableauSum}, so we can divide by it. 
\end{proof}

To relate theorem \ref{HLpositivity} to proposition \ref{PlacticPositivity} associate $v_{1} \otimes v_{2} \otimes \cdots \otimes v_{n}$ to $S \subseteq \{1, 2, \ldots, n\}$ by including $1 \leq i \leq n$ in $S$ if and only if $v_{i} = \mathsf{1}$. Then one can check that for $t=0$ operator $\widetilde{T}(\alpha)$ turns into action by $Id + \sum_{k=0} \alpha^{k} H_{k}$. Hence $T_{k}^{t=0} = H_{k}$ and proposition \ref{PlacticPositivity} becomes a special case of theorem \ref{HLpositivity}. See also \cite{BBW16}. 

\subsection{Paper outline} In section \ref{Kerov} we explore connection with a problem of Kerov on classifying homomorphisms from the algebra of symmetric functions to $\mathbb{R}$ with non-negative values on Macdonald functions. 
In section \ref{classicRSK} we recall the background on plactic algebra, Schensted's insertions and prove proposition \ref{PlacticPositivity}. In section \ref{tDeformations} we explore $t$-deformations of Schensted insertions and introduce extended vertex models into the picture. In section \ref{Proof} we prove theorem \ref{HLpositivity}. 

\section{Positive homomorphisms}
\label{Kerov}
Denote by $\Lambda$ the algebra of symmetric power series of bounded degree (called {\it symmetric functions}) in countably many variables $x_{1}, x_{2}, x_{3}, \ldots$ over $\mathbb{R}$. For fixed parameters $-1 < q, t < 1$ algebra $\Lambda$ admits two special linear bases of {\it Macdonald functions}: $\left \{P_{\lambda}(x_{1}, x_{2}, x_{3}, \ldots; q, t)  \right \}_{\lambda \in \mathcal{P}}$ and $\left \{Q_{\lambda}(x_{1}, x_{2}, x_{3}, \ldots; q, t)  \right \}_{\lambda \in \mathcal{P}}$. Here $\mathcal{P}$ denotes the set of \emph{partitions}, and  $Q_{\lambda}(q, t) = b_{\lambda}(q, t)P_{\lambda}(q, t)$ for some constant $b_{\lambda}(q, t) > 0$. See \cite{Mac99} for more background on Macdonald functions.
In particular, the one-row Macdonald functions are
\begin{align}
\label{GeneratingCoefficients0}
g_{r}:=Q_{(r)} = \sum_{r_{1}, r_{2}, r_{3}, \ldots \geq 0: \ r_{1}+r_{2}+r_{3}+\cdots = r} \prod_{i \geq 1} \frac{(t; q)_{r_{i}}}{(q, q)_{r_{i}}}x_{i}^{r_{i}}, \quad \text{where} \quad (a; q)_{k}:=\prod_{m=1}^{k}\left(1-aq^{m-1}\right)
\end{align}
is the $q$-{\it Pochhammer symbol}. Also $e_{r} = \sum_{1 < i_{1} < i_{2} < \cdots < i_{r}}x_{i_{1}}x_{i_{2}} \cdots x_{i_{r}}$, the $r$-th \emph{elementary symmetric function} is the same as $P_{1^{r}}$, the one-column Macdonald function. Note that for $q=0$ Macdonald functions $P_{\lambda}$ become the Hall-Littlewood functions, which turn into Hall-Littlewood polynomials $P_{\lambda}$ of subsection \ref{HL polynomials} after setting $x_{i} \to a_{i}$ for $1 \leq i \leq n$ and $x_{i} \to 0$ for $i > n$. Then $g_{r}$ of equality \eqref{GeneratingCoefficients0} becomes $g_{r}$ of equality \eqref{generatingfunction}. Any element of $\Lambda$ can be uniquely represented as a polynomial in $g_{r}$'s.
The following result was conjectured by S.V.~Kerov in \cite[Sec.~7.3]{Ker92} (see also \cite[p.~106]{Ker03}) and proved by the author in \cite{M19}.
\begin{theorem}[\cite{M19}]
\label{theorem:kerovconjecture}
For fixed $-1 < q, t < 1$ a homomorphism $\theta: \Lambda \to \mathbb{R}$ has the property that $\theta(P_{\lambda}) \geq 0$ for any partition $\lambda \in \mathcal{P}$ (is Macdonald-positive) if and only if it is defined by the generating function
\begin{align}
\label{positivehomsdef}
1+\sum_{r=1}^{\infty} \theta(g_{r})z^{r} = e^{\gamma z} \cdot \prod_{i=1}^{\infty} \frac{(t \alpha_{i} z; q)_{\infty}}{(\alpha_{i} z; q)_{\infty}} \cdot \prod_{j=1}^{\infty} \left(1 + \beta_{j}z \right)
\end{align}
for some $\alpha_{i}, \beta_{j}, \gamma \geq 0$, such that $\displaystyle \sum_{i=1}^{\infty} \alpha_{i} + \sum_{j=1}^{\infty} \beta_{j} < \infty$.
\end{theorem}
For $q=t$ both functions $P_{\lambda}$ and $Q_{\lambda}$ become the Schur function $S_{\lambda}$. The corresponding special case of Theorem \ref{theorem:kerovconjecture} is known as the \emph{Edrei-Thoma theorem}. It was first conjectured in \cite{Sch48} and then proved in a series of papers  \cite{ASW52}, \cite{Edr52}, \cite{Whit52} in the language of classifying infinite totally non-negative upper triangular Toeplitz matrices.  It was independently discovered and proved in \cite{Th64} in the context of classifying characters of the \emph{infinite symmetric group} $S_{\infty}$. See \cite{BO17} for more details on the representation theory of the infinite symmetric group.

\begin{corollary} 
\label{FiniteKerov}
Suppose $0 \leq t < 1$. Homomorphism $\theta:\Lambda_{n} \to \mathbb{R}$ has non-negative values on all Hall-Littlewood symmetric polynomials if and only if it is given by setting $\theta(a_{i}):= \alpha_{i} \geq 0$ and restricting to $\Lambda_{n}$. 
\end{corollary}
\begin{proof}[Proof of corollary \ref{FiniteKerov}]
$\theta$ defined by by setting $\theta(a_{i}):= \alpha_{i} \geq 0$ is non-negative on all Hall-Littlewood polynomials due to equality \eqref{HLTableauSum}. For the reverse direction first define $\pi: \Lambda \to \Lambda_{n}$ by setting $\pi(x_{i}) = a_{i}$ for $1 \leq i \leq n$, $\pi(x_{i}) = 0$ for $i > n$ and restricting to symmetric functions. $\pi(e_{r}) = e_{r}(a_{1}, a_{2}, \ldots, a_{n})$ for $1\leq r \leq n$, $\pi(e_{r}) = 0$ for $r > n$. 
There is a homomorphism $w_{t, 0}: \Lambda \to \Lambda$ sending the $t$-Whittaker function $P_{\lambda}(t, 0)$ to the Hall-Littlewood function $Q_{\lambda'}(0, t) = b_{\lambda'}(0, t)P_{\lambda'}(0, t)$ for a conjugate partition $\lambda'$, as well as sending $Q_{\lambda}(t, 0)$ to $P_{\lambda}(t, 0)$, see \cite{Mac99} for details. Suppose a homomorphism $\theta:\Lambda_{n} \to \mathbb{R}$ has non-negative values on all Hall-Littlewood symmetric polynomials.  Define $\tilde{\theta}: \Lambda \to \mathbb{R}$ by $\tilde{\theta} = \theta \circ \pi \circ w_{t, 0}$. Then $\tilde{\theta}$ takes non-negative values on all $P_{\lambda}(t, 0)$, so by theorem \ref{theorem:kerovconjecture} it can be defined by the generating function
\begin{align*}
1+\sum_{r=1}^{\infty} \tilde{\theta}(g_{r}(t, 0))z^{r} = e^{\gamma z} \cdot \prod_{i=1}^{\infty} \frac{1}{(\alpha_{i} z; t)_{\infty}} \cdot \prod_{j=1}^{\infty} \left(1 + \beta_{j}z \right)
\end{align*}
The left hand side of this equality is $1+\sum_{r=1}^{\infty} \tilde{\theta}(g_{r}(t, 0))z^{r} = 1 + \sum_{r=1}^{n}\theta(e_{r}(a_{1}, a_{2}, \ldots, a_{n}))z^{r}$. In particular, the coefficient of $z^{m}$ in it is $0$ for $m > n$. Note that we have  
\begin{align*}
\frac{1}{(\alpha_{i} z; t)_{\infty}} = \sum_{k=0}^{\infty}\frac{\alpha_{i}^{k}z^{k}}{(t; t)_{k}} 
\end{align*}
by the $t$-Gauss summation formula. So if either $\gamma > 0$, or at least one $\alpha_{i} > 0$, or there were more than $n$ non-zero $\beta_{j}$'s, then $1+\sum_{r=1}^{\infty} \tilde{\theta}(g_{r}(t, 0))z^{r}$ would contain powers of $z$ higher than $n$ with strictly positive coefficients. That would be a contradiction. Hence $$1 + \sum_{r=1}^{n}\theta(e_{r}(a_{1}, a_{2}, \ldots, a_{n}))z^{r} = \prod_{j=1}^{n} \left(1 + \beta_{j}z \right).$$ Relabel $\beta_{j}$ as $\alpha_{j}$. Then $\theta(e_{r}(a_{1}, a_{2}, \ldots, a_{n})) = e_{r}(\alpha_{1}, \alpha_{2}, \ldots, \alpha_{n})$.  So $\theta$ is defined by setting $\theta(a_{i}):= \alpha_{i} \geq 0$ and restricting to $\Lambda_{n}$. 

\end{proof}
So theorem \ref{HLpositivity} can be viewed as a multidimensional generalization of the easier direction of corollary \ref{FiniteKerov}. This raises a question: what would be the right multidimensional generalization of the harder direction?

\section{Plactic algebra and Schensted's insertions}
\label{classicRSK}
\subsection{Signatures, tableaux and particle arrays}
A \emph{signature}\footnote{These objects 
are also sometimes called \emph{highest weights} as they are the highest weights
of irreducible representations of the unitary group
$U(n)$.} of length $n\ge1$
is a non-increasing collection of integers 
$\la=(\la_1\ge \ldots\ge\la_n)\in\Z^{n}$.
We will work with signatures which have only nonnegative 
parts, i.e., $\la_n\ge0$ (in which case they are also called \emph{partitions}).
Denote the set of all such objects by $\GT_n$.
Let also $\GT:=\bigcup_{n=1}^{\infty}\GT_n$,
with the understanding that we identify
$\la\cup0=(\la_1,\ldots,\la_n,0,0,\ldots,0)\in\GT_{n+m}$
($m$ zeros)
with $\la\in\GT_n$
for any $m\ge1$.

We will use two ways to depict signatures (see Fig.~\ref{fig:YD}):
\begin{enumerate}
	\item Any signature $\la\in\GT_n$ 
	can be identified with a \emph{Young diagram}
	(having at most $n$ rows)
	as in \cite[I.1]{Mac99}.
	\item A signature $\la\in\GT_n$ can also be represented
	as a configuration of $n$ particles on $\Z_{\ge0}$
	(with the understanding that there can be more than one particle
	at a given location).
\end{enumerate}
We denote by $|\la|:=\sum_{i=1}^{n}\la_i$ the number of 
boxes in the corresponding Young diagram,
and by $\ell(\la)$ the number of nonzero parts in $\la$ 
(which is finite for all $\la\in\GT$).
For $\mu,\la\in\GT$ we
will write $\mu\subseteq\la$
if $\mu_i\le\la_i$ for all $i \in \Z_{>0}$. 
In this case, the set difference 
of Young diagrams $\la$ and $\mu$
is denoted by $\la/\mu$ and is called a \emph{skew Young diagram}.

Two signatures $\mu,\la\in\GT$ are said to \emph{interlace}
if one can append them by zeros such that
$\mu\in\GT_{n-1}$ and $\la\in\GT_n$
for some $N$, and
\begin{align}\label{interlace}
	\la_1\ge\mu_1\ge\la_2\ge\mu_2 \ge\ldots\ge
	\la_{n-1}\ge\mu_{n-1}\ge
	\la_{n}.
\end{align}
In terms of Young diagrams, this means that
$\la$ is obtained from $\mu$
by adding a \emph{horizontal strip} (or, equivalently, that
\emph{the skew diagram $\la/\mu$ is a horizontal strip} which is,
by definition, a skew Young diagram having at most one box in each 
vertical column),
and we denote this by $\mu\prech\la$.

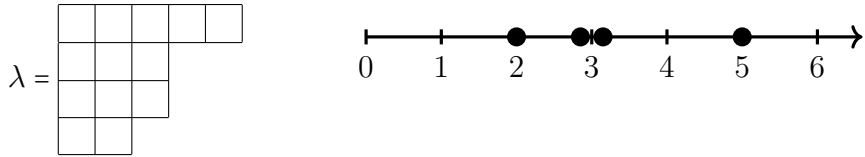
\begin{figure}[htbp]
\begin{equation*}
	\la=\begin{array}{|c|c|c|c|c|}
	    \hline
	    \ &\ &\ &\ &\ \\
	    \hline 
	    \ &\ &\ \\
	    \cline{1-3}
	    \ &\ &\ \\
		\cline{1-3}
		\ &\ \\
		\cline{1-2}
	\end{array}
	\hspace{40pt}
	\begin{tikzpicture}
		[scale=1, very thick]
		\draw[->] (0,0) -- (6.6,0);
		\foreach \h in {0,1,2,3,4,5,6}
		{
			\draw (\h,.1)--++(0,-.2) node [below] {$\h$};
		}
		\draw[fill] (2,0) circle (3pt);
		\draw[fill] (3-.15,0) circle (3pt);
		\draw[fill] (3+.15,0) circle (3pt);
		\draw[fill] (5,0) circle (3pt);
	\end{tikzpicture}
\end{equation*}
\caption{Young diagram $\la=(5,3,3,2)\in\GT_4$,
and the corresponding particle configuration.
Note that there are two particles at location 3.}
\label{fig:YD}
\end{figure}

Let $\la'$ denote the transposition of the 
Young diagram $\la$. For the diagram on
Fig.~\ref{fig:YD}, we have $\la'=(4,4,3,1,1)$.
If $\la/\mu$ is a horizontal strip, then $\la'/\mu'$
is called a \emph{vertical strip}.
We will denote the corresponding relation by $\mu' \precv\la'$.
\begin{definition}
A \emph{Gelfand--Tsetlin array} (sometimes also referred to as \emph{scheme}, or \emph{pattern}) of depth $n$ is a sequence of interlacing signatures
$\lab=(\la^{(1)}\prech\la^{(2)}\prech \ldots\prech \la^{(n)})$, where $\la^{(j)}\in\GT_j$.
\end{definition} 
Such sequences first appeared in connection with 
representation theory of unitary groups
\cite{GC50}.\footnote{This justifies the notation ``$\mathbb{GT}$'' we are using. Tsetlin and Cetlin are different English spellings of the same last name.}
We will depict sequences $\lab$ as interlacing integer arrays,
and also associate to them configurations of 
particles 
$\{(\la^{(k)}_j,k)\colon k=1,\ldots,n,\; j=1,\ldots,k\}$
on $n$ horizontal copies of $\Z_{\ge 0}$. See Fig.~\ref{fig:array}.
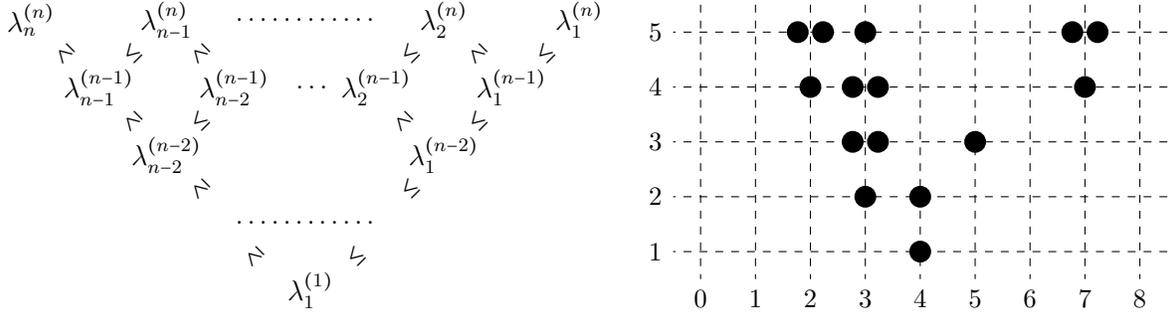
\begin{figure}[htbp]
	\begin{tabular}{cc}
	\begin{adjustbox}{max width=.48\textwidth}
		\begin{tikzpicture}[scale=1]
			\def\h{1}
			\def\x{1.9}
			\node at (-5,5) {$\la_n^{(n)}$};
			\node at (-3,5) {$\la_{n-1}^{(n)}$};
			\node at (0-\x/2,5) {$\ldots\ldots\ldots\ldots$};
			\node at (3-\x,5) {$\la_{2}^{(n)}$};
			\node at (5-\x,5) {$\la_{1}^{(n)}$};
			\node at (-4,5-\h) {$\la_{n-1}^{(n-1)}$};
			\node at (-2,5-\h) {$\la_{n-2}^{(n-1)}$};
			\node at (0-\x/2+.1,5-\h) {$\ldots$};
			\node at (2-\x,5-\h) {$\la_{2}^{(n-1)}$};
			\node at (4-\x,5-\h) {$\la_{1}^{(n-1)}$};
			\node at (-3,5-2*\h) {$\la_{n-2}^{(n-2)}$};
			\node at (3-\x,5-2*\h) {$\la_{1}^{(n-2)}$};
			\foreach \LePoint in {(4.5-\x,5-\h/2),(2.5-\x,5-\h/2),(-3.5,5-\h/2),
			(-2.5,5-3*\h/2),(3.5-\x,5-3*\h/2),(2.5-\x,5-5*\h/2),(1.7-\x,5-7*\h/2)} {
				\node [rotate=45] at \LePoint {$\le$};
			};
			\foreach \GePoint in {(3.5-\x,5-\h/2),(-4.5,5-\h/2),(-2.5,5-\h/2),
			(-3.5,5-3*\h/2),(-2.5,5-5*\h/2),(.5,5-3*\h/2),(-1.7,5-7*\h/2)} {
				\node [rotate=135] at \GePoint {$\ge$};
			};
			\node at (0-\x/2,5-3*\h) {$\ldots\ldots\ldots\ldots$};
			\node at (0-\x/2+.08,5-4*\h) {$\la_1^{(1)}$};
		\end{tikzpicture}
	\end{adjustbox}
	&
	\begin{tikzpicture}[scale = .73]
			\tikzset{every node/.style={
				font=\footnotesize
			}}
			\foreach \yy in {1,2,3,4,5}
			{
				\draw[dashed] (8.5, \yy) -- (-0.5,\yy) node[left] {\yy};
			};
			\foreach \xxy in {0,1,2,3,4,5,6,7,8}
			{
				\draw[dashed]  (\xxy,5.5) -- (\xxy,.5) node[below] {\xxy};
			};
			\def\x{0.16}
			\def\y{0.23}
			\def\ys{-1}
			\def\xs{7}
			\def\xss{-3}
			\draw[ultra thick, fill] (4,1) circle(\x);
			\draw[ultra thick, fill] (3,2) circle(\x);
			\draw[ultra thick, fill] (4,2) circle(\x);
			\draw[ultra thick, fill] (3-\y,3) circle(\x);
			\draw[ultra thick, fill] (3+\y,3) circle(\x);
			\draw[ultra thick, fill] (5,3) circle(\x);
			\draw[ultra thick, fill] (2,4) circle(\x);
			\draw[ultra thick, fill] (3-\y,4) circle(\x);
			\draw[ultra thick, fill] (3+\y,4) circle(\x);
			\draw[ultra thick, fill] (7,4) circle(\x);
			\draw[ultra thick, fill] (2-\y,5) circle(\x);
			\draw[ultra thick, fill] (2+\y,5) circle(\x);
			\draw[ultra thick, fill] (3,5) circle(\x);
			\draw[ultra thick, fill] (7-\y,5) circle(\x);
			\draw[ultra thick, fill] (7+\y,5) circle(\x);
		\end{tikzpicture}
	\end{tabular}
	\caption{Left: An interlacing array $\lab$; 
	we require that $\la^{(j)}_{i}\in\Z_{\ge0}$.
	Right:
	A configuration of particles corresponding to
	an interlacing array of depth $n=5$ (right).}
	\label{fig:array}
\end{figure}

Let us denote the set of all interlacing
arrays $\lab$ of depth $n$ with top level $\la$ by $\mathbb{GT}^{(n)}(\la)$. Let $\mathbb{GT}^{(n)}:= \bigcup_{\la \in \GT_n}\mathbb{GT}^{(n)}(\la)$.
\begin{definition}
A \emph{semistandard Young tableau} of shape $\la$ is a filling in the boxes of the Young diagram $\la$ with positive integers, which increase weakly along rows, and strictly down columns.
\end{definition}
There is a natural correspondence between the Gelfand-Tsetlin arrays of depth $n$ and the semistandard Young tableaux filled with numbers from $1$ to $n$. Indeed, given $\lab \in \mathbb{GT}^{(n)}$ we can produce a semistandard Young tableau of shape $\la^{(n)}$ by filling $\la^{(j)}/\la^{(j-1)}$ with numbers equal to $j$, see Fig.~\ref{fig:SSYT}. Thus, by a slight abuse of notation we will also use $\mathbb{GT}^{(n)}(\la)$ to denote the set of semistandard Young tableaux of shape $\la$ filled with numbers from $1$ to $n$.

\begin{figure}[htbp]
	\begin{adjustbox}{max height=.2\textwidth}
	\begin{tikzpicture}[scale = .5, thick]
		\draw (0,0) --++ (7,0);
		\draw (0,-1) --++ (7,0);
		\draw (0,-2) --++ (7,0);
		\draw (0,-3) --++ (3,0);
		\draw (0,-4) --++ (2,0);
		\draw (0,-5) --++ (2,0);
		\draw (0,0) --++ (0,-5);
		\draw (1,0) --++ (0,-5);
		\draw (2,0) --++ (0,-5);
		\draw (3,0) --++ (0,-3);
		\draw (4,0) --++ (0,-2);
		\draw (5,0) --++ (0,-2);
		\draw (6,0) --++ (0,-2);
		\draw (7,0) --++ (0,-2);
		\def\xxxx{.05}
		\node[anchor=west] at (\xxxx,-.5) {1};
		\node[anchor=west] at (\xxxx+1,-.5) {1};
		\node[anchor=west] at (\xxxx+2,-.5) {1};
		\node[anchor=west] at (\xxxx+3,-.5) {1};
		\node[anchor=west] at (\xxxx+4,-.5) {3};
		\node[anchor=west] at (\xxxx+5,-.5) {4};
		\node[anchor=west] at (\xxxx+6,-.5) {4};
		\node[anchor=west] at (\xxxx,-1.5) {2};
		\node[anchor=west] at (\xxxx+1,-1.5) {2};
		\node[anchor=west] at (\xxxx+2,-1.5) {2};
		\node[anchor=west] at (\xxxx+3,-1.5) {5};
		\node[anchor=west] at (\xxxx+4,-1.5) {5};
		\node[anchor=west] at (\xxxx+5,-1.5) {5};
		\node[anchor=west] at (\xxxx+6,-1.5) {5};
		\node[anchor=west] at (\xxxx,-2.5) {3};
		\node[anchor=west] at (\xxxx+1,-2.5) {3};
		\node[anchor=west] at (\xxxx+2,-2.5) {3};
		\node[anchor=west] at (\xxxx,-3.5) {4};
		\node[anchor=west] at (\xxxx+1,-3.5) {4};
		\node[anchor=west] at (\xxxx,-4.5) {5};
		\node[anchor=west] at (\xxxx+1,-4.5) {5};
	\end{tikzpicture}
	\end{adjustbox}
	\caption{A semistandard Young tableau corresponding
	to the array on Fig.~\ref{fig:array}, right.}
	\label{fig:SSYT}
\end{figure}
\subsection{Schensted's insertions and interacting particle systems}

Schensted's row and column insertions (\cite{Sch61}) are combinatorial constructions serving as the building blocks of the RSK algorithms, see \cite{K70}, \cite{F97}. Each insertion can be described in the language of semistandard Young tableaux as a sequence of row and column bumpings. In the language of interlacing arrays these bumpings correspond to elementary operations
of deterministic long-range pulling and pushing, which involve only two consecutive levels of an array.
\begin{definition}(Deterministic long-range pulling, Fig.~\ref{fig:pulling})
	\label{def:pull}
	\smallskip
	
	Let $j=2,\ldots,n$, 
	and signatures 
	$\bar\la,\bar\nu\in\GT_{j-1}$, 
	$\la\in\GT_{j}$ satisfy $\bar\la\prech \la$, 
	$\bar\nu=\bar\la+\bar{\mathrm{e}}_{i}$,
	where $\bar{\mathrm{e}}_{i}=(0,0,\ldots,0,1,0,\ldots,0)$ (for some $i=1,\ldots,j-1$) 
	is the $i$th basis vector of length $j-1$.
	Define $\nu\in\GT_j$ to be
	\begin{align*}
		\nu=\mathsf{pull}(\la\mid \bar\la\to\bar\nu):=
		\begin{cases}
			\la+\mathrm{e}_{i},&\text{if $\bar\la_{i}=\la_i$};\\
			\la+\mathrm{e}_{i+1},&\text{otherwise}.
		\end{cases}
	\end{align*}
	Here $\mathrm{e}_{i}$ and $\mathrm{e}_{i+1}$
	are basis vectors of length $j$.

	In words, the particle $\bar\la_i$
	at level $j-1$  which moved to the right by one
	generically pulls its upper left neighbor $\la_{i+1}$, or 
	pushes it upper right neighbor $\la_i$
	if the latter operation is needed to preserve the interlacing.
	Note that the long-range pulling mechanism
	does not encounter any blocking issues.
\end{definition}
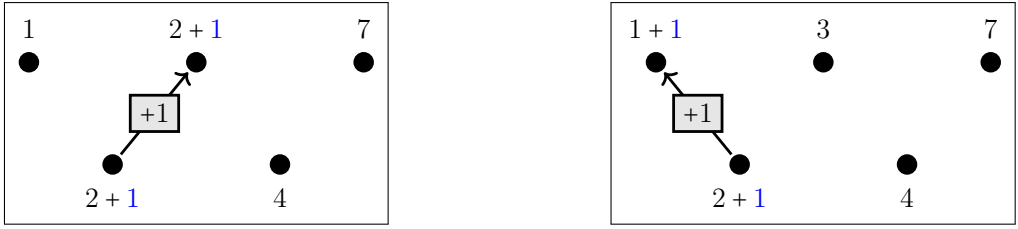
\begin{figure}[htbp]
	\begin{tabular}{cc}
	\framebox{\begin{adjustbox}{max height=.16\textwidth}\begin{tikzpicture}
		[scale=.7, thick]
		\def\cir{.2}
		\def\ysh{6}
		\def\x{1.8}
		\def\y{2.2}
		\draw[fill] (4*\x,0) circle(\cir) node [below,yshift=-\ysh] {$2+{\color{blue}1}$};
		\draw[fill] (6*\x,0) circle(\cir) node [below,yshift=-\ysh] {$4$};
		\draw[fill] (3*\x,1*\y) circle(\cir) node [above,yshift=\ysh] {$1$};
		\draw[fill] (5*\x,1*\y) circle(\cir) node [above,yshift=\ysh] {$2+{\color{blue}1}$};
		\draw[fill] (7*\x,1*\y) circle(\cir) node [above,yshift=\ysh] {$7$};
		\node (jb) at (4*\x,0) {};
		\node (pa) at (5*\x,\y) {};
		\draw[->,very thick] (jb) -- (pa)
		node[rectangle,draw=black,fill=gray!20!white] [midway] {$+1$};
	\end{tikzpicture}
	\end{adjustbox}}
	\hspace{70pt}
	&
	\framebox{\begin{adjustbox}{max height=.16\textwidth}\begin{tikzpicture}
		[scale=.7, thick]
		\def\cir{.2}
		\def\ysh{6}
		\def\x{1.8}
		\def\y{2.2}
		\draw[fill] (4*\x,0) circle(\cir) node [below,yshift=-\ysh] {$2+{\color{blue}1}$};
		\draw[fill] (6*\x,0) circle(\cir) node [below,yshift=-\ysh] {$4$};
		\draw[fill] (3*\x,1*\y) circle(\cir) node [above,yshift=\ysh] {$1+{\color{blue}1}$};
		\draw[fill] (5*\x,1*\y) circle(\cir) node [above,yshift=\ysh] 
		{$3$};
		\draw[fill] (7*\x,1*\y) circle(\cir) node [above,yshift=\ysh] {$7$};
		\node (jb) at (4*\x,0) {};
		\node (pa) at (3*\x,\y) {};
		\draw[->,very thick] (jb) -- (pa)
		node[rectangle,draw=black,fill=gray!20!white] [midway] {$+1$};
	\end{tikzpicture}\end{adjustbox}}
	\end{tabular}
	\caption{An example of pulling mechanism
	for $i=2$
	at levels 2 and 3 (i.e., $j=3$).
	Left: $\bar\la_2=\la_2$,
	which forces the pushing of the upper right neighbor.
	Right: in the generic
	situation $\bar\la_2<\la_2$ the upper left neighbor is pulled.}
	\label{fig:pulling}
\end{figure}

\begin{definition}(Deterministic long-range pushing, Fig.~\ref{fig:pushing})
	\label{def:push}
	As in the previous definition, 
	let
	$j=2,\ldots,n$, 
	$\bar\la,\bar\nu\in\GT_{j-1}$, 
	$\la\in\GT_{j}$ be such that $\bar\la\prech \la$
	and
	$\bar\nu=\bar\la+\bar{\mathrm{e}}_{i}$.
	Define $\nu\in\GT_j$ to be
	\begin{align*}
		\nu=\mathsf{push}(\la\mid \bar\la\to\bar\nu):=
		\la+\mathrm{e}_{m},
		\qquad
		\text{where
		$m=\max\{p\colon 
		1\le p\le i\text{ and }\la_{p}<\bar\la_{p-1}\}$}.
	\end{align*}

	In words, the particle $\bar\la_i$ 
	at level $j-1$
	which moved to the right by one,
	pushes its first upper right neighbor $\la_m$
	which is not blocked (and therefore is free to move
	without violating the interlacing).
	Generically (when all particles are sufficiently
	far apart) $\la_m=\la_i$, so the immediate upper right
	neighbor is pushed.
\end{definition}
\begin{remark}[Move donation]\label{rmk:move_donation}
	It is useful to equivalently 
	interpret the mechanism of Definition
	\ref{def:push} in a slightly different way.
	Namely, let us say that 
	when the particle $\bar\la_i$
	at level $j-1$ moves, it gives the 
	particle $\la_i$ at level $j$ a \emph{moving impulse}.
	If $\la_i$ is blocked (i.e., if $\la_i=\bar\la_{i-1}$),
	this moving impulse is \emph{donated} to the next
	particle $\la_{i-1}$ to the right of $\la_i$. If $\la_{i-1}$
	is blocked, too, then the impulse is donated further, and so on.
	Note that the particle $\la_{1}$ cannot be blocked, so
	this moving impulse will always result in an actual move.
\end{remark}

\begin{figure}[htbp]
	\begin{adjustbox}{max height=.17\textwidth}
	\begin{tikzpicture}
		[scale=.7, thick]
		\def\cir{.2}
		\def\ysh{6}
		\def\x{1.8}
		\def\y{2.2}
		\draw[fill] (0,0) circle(\cir) node [below,yshift=-\ysh] {$1$};
		\draw[fill] (2*\x,0) circle(\cir) node [below,yshift=-\ysh] {$2+
		{\color{blue}1}$};
		\draw[fill] (4*\x,0) circle(\cir) node [below,yshift=-\ysh] {$4$};
		\draw[fill] (6*\x,0) circle(\cir) node [below,yshift=-\ysh] {$6$};
		\draw[fill] (-1*\x,1*\y) circle(\cir) node [above,yshift=\ysh] 
		{$0$};
		\draw[fill] (1*\x,1*\y) circle(\cir) node [above,yshift=\ysh] 
		{$1$};
		\draw[fill] (3*\x,1*\y) circle(\cir) node [above,yshift=\ysh] 
		{$4$};
		\draw[fill] (5*\x,1*\y) circle(\cir) node [above,yshift=\ysh-1.9] {$6$};
		\draw[fill] (7*\x,1*\y) circle(\cir) node [above,yshift=\ysh-1.9] {$7+{\color{blue}1}$};
		\node at (10*\x,0) {$\bar\la+({\color{blue}\bar\nu-\bar\la})$};
		\node at (10*\x,\y) {$\la+({\color{blue}\nu-\la})$};
		\node (sb) at (4*\x,0) {};
		\node (ba) at (3*\x,\y) {};
		\node (sb1) at (6*\x,0) {};
		\node (ba1) at (5*\x,\y) {};
		\node (la1) at (7*\x,\y) {};
		\node (lab3) at (2*\x,0) {};
		\draw[->,very thick] (lab3) .. controls 
		(4*\x,-\y/3)
		and
		(5*\x,4/3*\y)
		.. (la1) 
		node[rectangle,draw=black,fill=gray!20!white] [midway] {$+1$};
		\draw[very thick,dotted] (sb) -- (ba)
		node[rectangle,draw=black,fill=gray!20!white] [midway] {block};
		\draw[very thick,dotted] (sb1) -- (ba1)
		node[rectangle,draw=black,fill=gray!20!white] [midway] {block};
	\end{tikzpicture}
	\end{adjustbox}
	\caption{An example of pushing mechanism
	for $i=3$ at levels 4 and 5 
	(i.e., $j=5$). Since the particles
	$\la_3=\bar\la_2$ and $\la_2=\bar\la_1$
	are blocked, the first particle 
	which can be pushed is $\la_1$.}
	\label{fig:pushing}
\end{figure}
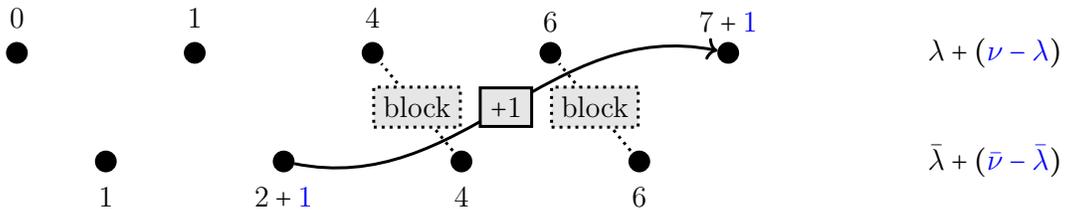
\begin{definition}The {\it Schensted's row insertion} is an algorithm that takes a semistandard tableau $\lab \in \mathbb{GT}^{(n)}$, and an integer $1 \leq x \leq n$, and constructs a new tableau $\lab \leftarrow x$ according to the following procedure:
 
$\bullet$ If $x$ is at least as large as all the entries in the first row of $\lab$, add $x$ in a new box to the end of the first row. In this case the algorithm terminates.

$\bullet$ Otherwise find the leftmost entry $y$ in the first row that is strictly larger than $x$ and replace it by $x$. 

$\bullet$ Repeat the same steps with $y$ and the second row, then with the replaced entry $z$ and the third row, ...,  and so on until the replaced entry can be put in the end of the next row, possibly by forming a new row of one entry. Then the algorithm terminates.
\end{definition}

In terms of arrays and long-range pulling we can describe this insertion in the following way:

$\bullet$ Levels $\la^{(1)}, \ldots, \la^{(x-1)}$ remain unchanged.

$\bullet$ Rightmost particle on the level $x$ moves by $1$ to the right, i.e $\la^{(x)} \to \la^{(x)} + \bar{\mathrm{e}}_{1}$.

$\bullet$ Then $\mathsf{pull}$ operations are consecutively performed for $j=x+1, \ldots, n$.
\begin{figure}[h]

\includegraphics[width = 0.9\textwidth]{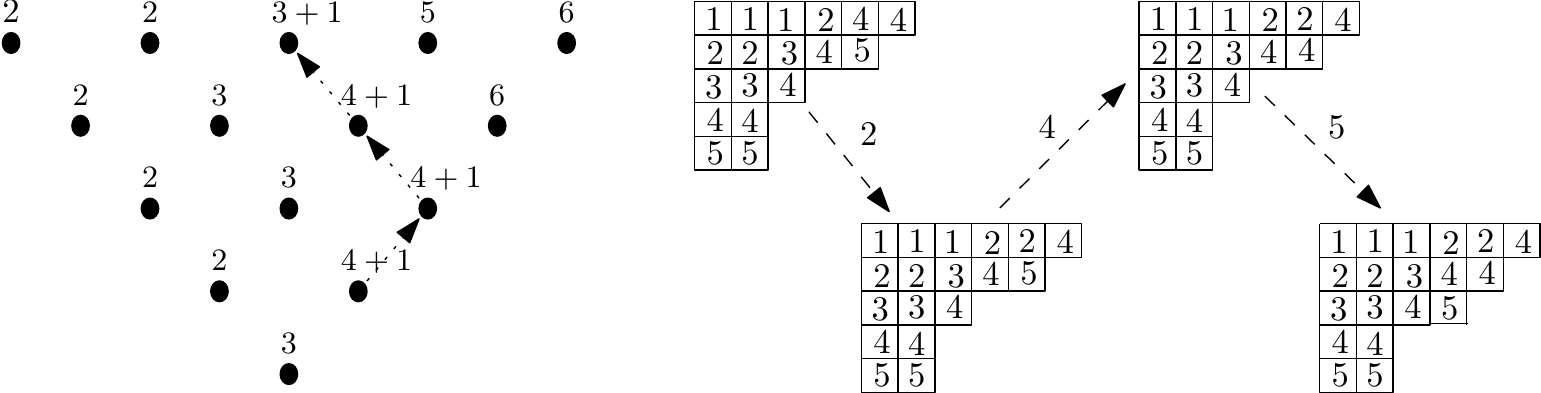}

\caption{An example of Schensted's row insertion in terms of semistandard tableaux and particle arrays for $n=5$.}

\end{figure} 

In words, this push-pull chain of movements starts on the right edge of the array and progresses upwards until it reaches the top level. This is the same as saying that shape of a tableau is augmented by one cell after row insertion of a single entry. One can also row-insert a word $X = x_{1}x_{2}\ldots x_{\ell}$ in a tableau $\lab$ by consecutively inserting its entries one by one:
\begin{align*}
\lab \leftarrow X := \lab \leftarrow x_{1} \leftarrow x_{2} \leftarrow \cdots \leftarrow x_{\ell}
\end{align*}

\begin{definition} The {\it Schensted's column insertion} is an algorithm that takes a semistandard tableau $\lab \in \mathbb{GT}^{(n)}$, and an integer $1 \leq x \leq n$, and constructs a new tableau $x \to \lab$ according to the following procedure:

$\bullet$ If $x$ is strictly larger than all the entries in the first column of $\lab$, add $x$ in a new box at the bottom  of the first column. In this case the algorithm terminates.

$\bullet$ Otherwise find the topmost entry $y$ in the first column that is at least large as $x$ and replace it by $x$. 

$\bullet$ Repeat the same steps with $y$ and the second column, then with the replaced entry $z$ and the third column, ...,  and so on until the replaced entry can be put at the bottom of the next column, possibly by forming a column of one entry. Then the algorithm terminates.

\end{definition}

In terms of arrays and long-range pushing we can describe this insertion in the following way:

$\bullet$ Levels $\la^{(1)}, \ldots, \la^{(x-1)}$ remain unchanged.

$\bullet$ Leftmost particle on the level $x$ moves by $1$ to the right, i.e $\la^{(x)} \to \la^{(x)} + \bar{\mathrm{e}}_{x}$.

$\bullet$ Then $\mathsf{push}$ operations are consecutively performed for $j=x+1, \ldots, n$.

\begin{figure}[h]

\includegraphics[width = 0.9\textwidth]{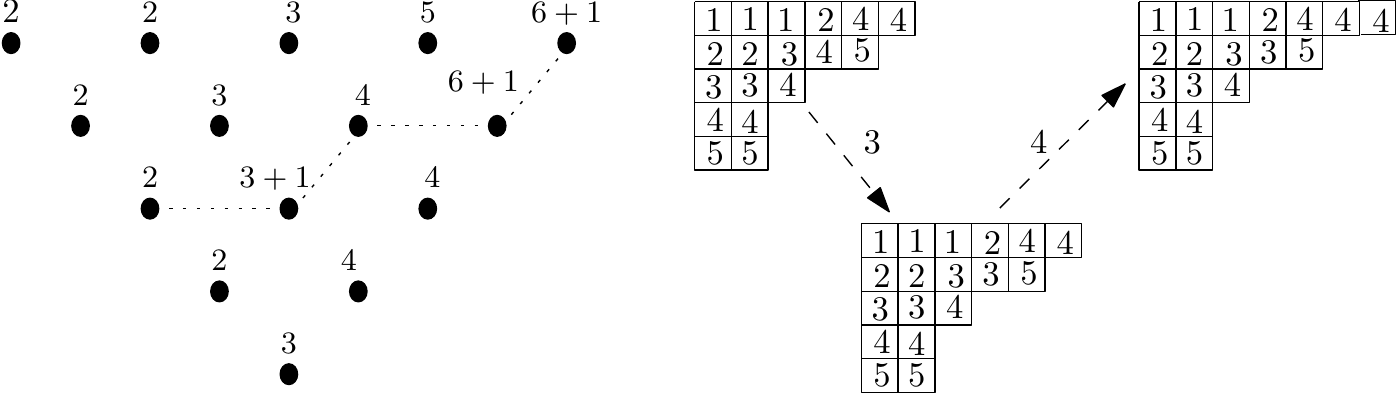}

\caption{An example of Schensted's column insertion in terms of semistandard tableaux and particle arrays for $n=5$. Only steps that change the tableau are shown.}

\end{figure} 

As in the case of the row insertion, on each of the levels from $x$-th to $N$-th precisely one particle moves to the right by $1$. The sequence of moves progresses upwards and to the right until it reaches the top level. One can also column-insert a word $X = x_{1}x_{2}\ldots x_{\ell}$ in a tableau $\lab$ by consecutively inserting its entries one by one in reverse order:
\begin{align*}
X \to \lab := x_{1} \to x_{2} \to  \cdots \to x_{\ell} \to \lab
\end{align*}

\subsection{Plactic Algebra}
To a semistandard tableau $\lab \in \mathbb{GT}^{(n)}$ one can associate an element $w(\lab) \in Pl_{n}$  represented by a word obtained by reading entry letters of $\lab$ first in the bottom row from left to right, then in the second  row  from the bottom from left to right, and so on. For instance, for a tableau on Fig.~\ref{fig:SSYT} the corresponding word will be $\mathsf{554433322255551111344}$.  The following proposition explains basic connections between the plactic monoid and Schensted insertions.

\begin{proposition} (see \cite{Lot02}).

\begin{enumerate}

\item
For every $a \in Pl_{n}$ there exists a unique tableau $\lab$, such that $a = w(\lab)$.

\item
$w(\lab \leftarrow x) = w(\lab) x$.

\item
$w(x \to \lab) = x w(\lab)$.

\end{enumerate}

\end{proposition}

Hence consecutively inserting $x_{1}, x_{2}, \ldots, x_{r}$ via the Schensted's row insertion into a tableau $\lab$ leads to the the same result as multiplication of $w(\lab)$ by $X = x_{1}x_{2}\cdots x_{r}$ on the right, while multiplication of $w(\lab)$ by $X$ on the left amounts to the same result as consecutively inserting $x_{r}, \ldots, x_{2}, x_{1}$ in $\lab$ via the Schensted's column insertion. 

\medskip

The plactic algebra is noncommutative for $n \ge 2$, but it contains nice families of commuting elements. More precisely, for $\la \in \GT_n$ and variables $a_{1}, a_{2}, \ldots, a_{n}$ define the {\it plactic Schur polynomial}
\begin{align}
S_{\la}^{Pl}(a_{1} \mathsf{1}, a_{2} \mathsf{2}, ,\ldots, a_{n} \mathsf{n}):= \sum_{\lab \in \mathbb{GT}^{(n)}(\la)} w(\lab) \cdot a_{1}^{|\la^{(1)}|}a_{2}^{|\la^{(2)}|-|\la^{(1)}|} \cdots a_{n}^{|\la^{(n)}|-|\la^{(n-1)}|}.
\end{align}
 
\begin{proposition} 
(see \cite{Lot02}). $S_{\la}(a_{1} \mathsf{1} ,\ldots, a_{n} \mathsf{n})$ and  $S_{\mu}(a_{1}  \mathsf{1},\ldots, a_{n} \mathsf{n})$ commute for arbitrary $\la$ and $\mu$, and their product can be expressed as 
\begin{align}
\label{plactic_expansion}
S_{\la}^{Pl}(a_{1} \mathsf{1},\ldots, a_{n} \mathsf{n}) S_{\mu}^{Pl}(a_{1} \mathsf{1},  \ldots, a_{n} \mathsf{n})  = \sum_{\nu} c^{\nu}_{\la, \mu} S_{\nu}^{Pl}(a_{1} \mathsf{1},\ldots, a_{n} \mathsf{n}),
\end{align}
where  $c^{\nu}_{\la, \mu}$ is the Littlewood-Richardson coefficient, i.e the coefficient of $S_{\nu}$ in the expansion of $S_{\la}S_{\mu}$ in the basis of Schur functions. 
\label{prop:comsub}
\end{proposition}

\begin{remark} A homomorphism from the plactic algebra to $\Lambda_{n}$ defined by sending every generator 
$\mathsf{k}$ to $1$ sends \eqref{plactic_expansion} to
\begin{align*}
S_{\la}(a_{1},\ldots, a_{N}) S_{\mu}(a_{1}, \ldots, a_{N}) = \sum_{\nu} c^{\nu}_{\la, \mu} S_{\nu}(a_{1},\ldots, a_{N}),
\end{align*}
which is the defining relation for the Littlewood-Richardson coefficients (for $\ell(\lambda), \ell(\mu), \ell(\nu) \leq n$). 
\end{remark}
We will now simplify notation by identifying $\lab$ with $w(\lab)$. 

\subsection{Plactic algebra action continued}
\label{PlacticAlgebraActionsContinued}
We are now ready to continue subsection \ref{PlacticAlgebraActions} and prove proposition \ref{PlacticPositivity}. 

\begin{proof}[Proof of proposition \ref{PlacticPositivity}]
For a tableau $\lab$ denote by $\mathcal{S}(\lab)$ the set of entries of the first column of $\lab$. Extend this function linearly to the whole plactic algebra. It follows from the description of Schensted's insertion that $\mathcal{S}(\tilde{\lab } \lab)$ depends on $\tilde{\lab}$ and $S(\lab)$, but not on the whole $\lab$. Identify $S \subseteq \{1, 2, \ldots, n\}$ with a one-column tableau with $S$ as the set of entries. Then it follows from the description of Schensted's insertion that $H_{r} \cdot S = \mathcal{S}(S^{Pl}_{(r)} \cdot S)$. Here $S^{Pl}_{(r)}$ is the corresponding one-row plactic Schur polynomial. Since $S^{Pl}_{(r)}$'s commute with each other, we can apply the Jacobi-Trudi formula  to get
\begin{align}
\label{PlacticTableauSum}
\det 
\left[ H_{\lambda_{i} - i + j} \right]_{i, j=1}^{\ell} \cdot S = S^{Pl}_{\lambda} \cdot S = \sum_{\lab}  a_{1}^{|\la^{(1)}|}a_{2}^{|\la^{(2)}|-|\la^{(1)}|} \cdots a_{n}^{|\la^{(n)}|-|\la^{(n-1)}|} \mathcal{S}(\lab S)
\end{align}
Note that here $\cdot$ denotes plactic algebra action as defined in subsection \ref{PlacticAlgebraActions}, while multiplication in the plactic algebra itself is written without a dot. Equality \eqref{PlacticTableauSum} implies positivity in proposition \ref{PlacticPositivity}. 
\end{proof}
Note that for the special case $S = \{1, 2, \ldots, n\}$ 
equality \eqref{PlacticTableauSum} becomes equality \eqref{BasicTableauSum}. 

\section{Searching for a $t$-deformation of plactic action}
\label{tDeformations}
\subsection{Towards proving theorem \ref{HLpositivity}}
Proof of proposition \ref{PlacticPositivity} together with equality \eqref{HLTableauSum} suggest the following plan to prove theorem \ref{HLpositivity}. For a semistandard tableau $\lab \in \mathbb{GT}^{(n)}$ find a linear operator $T_{\lab}:V^{\otimes n} \to V^{\otimes n}$ such that
\begin{enumerate}
\item
$T_{\lab}^{t=0}(S) = \mathcal{S}(\lab S)$

\item 
Matrix elements of $T_{\lab}$ with respect to basis $\{\mathsf{1}, \mathsf{2}\}^{\otimes n}$ are positive for $0 \leq t <1$. 

\medskip

\item
\begin{align}
\label{HLPlacticSum}
\Theta(P_{\lambda}) = \sum_{\lab \text{ of shape } \lambda} \psi_{\lab}(t)a^{\lab} T_{\lab}
\end{align}
\end{enumerate}
If we were to find such operators, the positivity in theorem \ref{HLpositivity} would follow, just as proposition \ref{HLBasicPositivity} follows from equality \eqref{HLTableauSum}. For a one-row $\lab$ with entries $i_{1} \leq i_{2} \leq \cdots \leq i_{r}$ we must take the coefficient of $a_{i_{1}}a_{i_{2}}\cdots a_{i_{r}}$ in $\frac{1}{1-t}T_{r}$ in order to satisfy condition \eqref{HLPlacticSum}. Similarly, for a one-column $\lab$ with entries $i_{1} < i_{2} < \cdots < i_{r}$ we must take the coefficient of $a_{i_{1}}a_{i_{2}}\cdots a_{i_{r}}$ in $\Theta(e_{r})$. However, for $\lab$ with more rows and columns the choice and existence of $T_{\lab}$ are not clear. The idea is think in line with some of the previous works on deformations of RSK algorithms, i.e. \cite{OP13}, \cite{MP17}, \cite{BM18}. A feature of this algorithms is that at intermediate stages we need to store additional information. This and experimentation suggests that $T_{\lab}$ should act by $t$-inserting columns of $\lab$ in the reverse order and preserving information about previous insertions. This idea together with guess and check leads to the following

\subsection{Extended vertex models enter the picture}
To prove theorem \ref{HLpositivity} we will derive representation of $\Theta(P_{\lambda})$ from which the desired positivity is evident. This is accomplished by introducing the following  ("extended") $3$-colored vertex model.  Let $W$ be a three-dimensional real vector space spanned by elements $\mathsf{0}, \mathsf{1}, \mathsf{2}$. Let $i: V \hookrightarrow W$ be the natural inclusion and $\pi: W \to V$ be a projection defined via $\pi(\mathsf{0}) = \mathsf{1}$. These maps, respectively, induce inclusion $i^{\otimes n} :V^{\otimes n} \hookrightarrow W^{\otimes n}$ and projection $\pi^{\otimes n} :W^{\otimes n} \to V^{\otimes n}$.  $W^{\otimes n} = \bigoplus_{k=0}^{n}W_{n, k}$, where each subspace $W_{n, k}$ is defined as a span of vectors  $e_{1} \otimes e_{2} \otimes \cdots \otimes e_{n}$ with each $e_{i} \in \{ \mathsf{0}, \mathsf{1}, \mathsf{2} \}$ and exactly $k$ $\mathsf{0}$'s among the $e_{i}$'s. Let $U_{2}$ be an (infinite-dimensional) real vector space of finite formal linear combinations of pairs $\left \{ \{ x, y \} \in \mathbb{Z}_{\geq 0}^{2} \right \}$. Given two parameters $a, t$ we define an operator $R_{ext} = R_{ext}(a, t): W \otimes U_{2} \to W \otimes U_{2}$ by 
\begin{flalign}
\label{ExtendedVertexDef}
& R_{ext}(\mathsf{0} \otimes \{ x, y \}) =    a \cdot \mathsf{0} \otimes \{ x, y \} + \left(t^{y} - t^{x+y} \right) \cdot \mathsf{1} \otimes \{ x-1, y \} + \left(1 - t^{y} \right) \cdot \mathsf{2} \otimes \{ x, y-1 \}, \nonumber \\ 
& R_{ext}(\mathsf{1} \otimes \{ x, y \}) = a \cdot \mathsf{0} \otimes \{ x+1, y \} + t^{y}  \cdot \mathsf{1} \otimes \{ x, y \}  + \left(1 - t^{y} \right) \cdot \mathsf{2} \otimes \{x+1, y-1 \}, \\
& R_{ext}(\mathsf{2} \otimes \{ x, y \}) = a \cdot \mathsf{0} \otimes \{ x, y+1 \} +  \mathsf{2} \otimes \{ x, y \} \nonumber. 
\end{flalign}
$R_{ext}$ gives rise to a vertex model with weights as specified on Fig. \ref{ExtendedVertex}. 
\begin{figure}[h]
\includegraphics[width = 
0.9\textwidth]{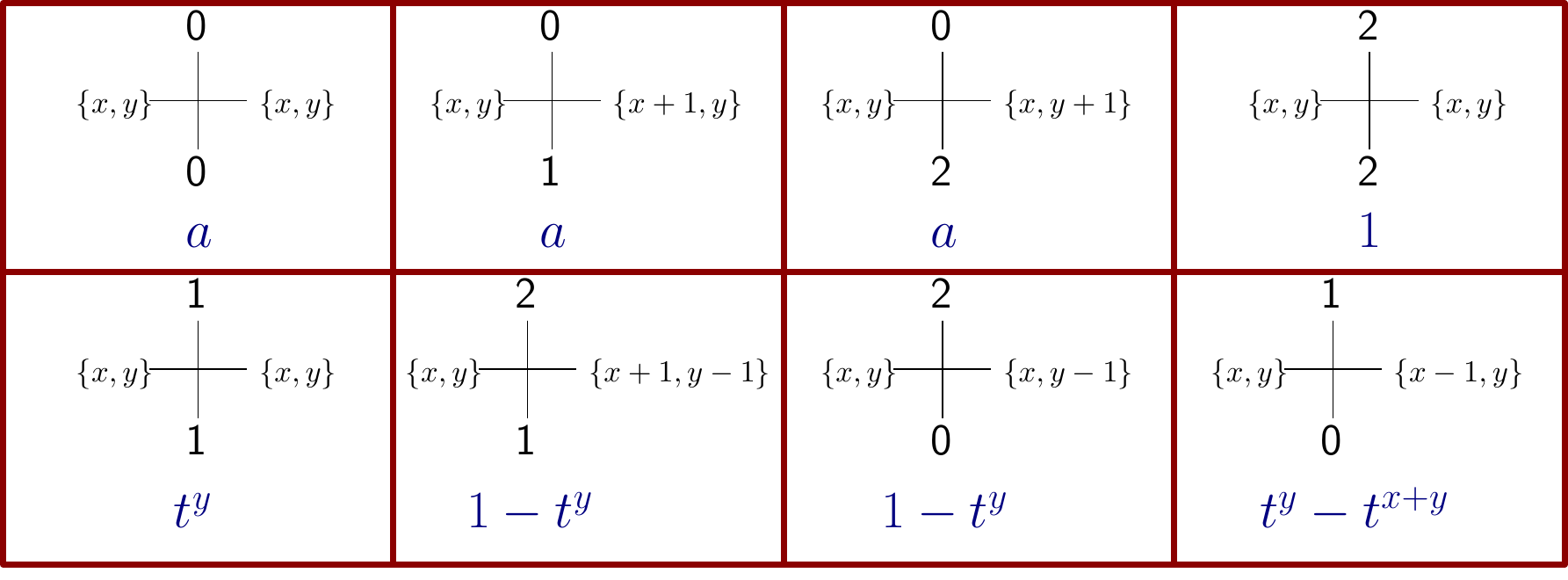}
\caption{Weights of the vertex model defined via $R_{ext}(a, t)$. All configurations not appearing on this picture are assumed to have weight $0$.}
\label{ExtendedVertex}
\end{figure} 

Define an operator $H: W^{\otimes n} \to W^{\otimes n}$ as the inhomogeneous transfer operator of $R_{ext}$ as specified on Fig. \ref{ExtendedTransferMatrix}. Note that there is fixed input $\{0, 0\}$ on the left, while boundary condition on the right is free. $H$ depends on parameters $t, a_{1}, a_{2}, \ldots, a_{n}$. If a pair $\{x, y\} \in \mathbb{Z}_{\geq 0}^{2}$ appears on the right boundary of the non-zero weight configuration with bottom row $e_{1} \otimes e_{2} \otimes \cdots \otimes e_{n}$ and top row $e_{1}' \otimes e_{2}' \otimes \cdots \otimes e_{n}'$, then it is clear from the model that 
\begin{align*}
x = \text{number of $\mathsf{1}$'s among the $e_{i}$'s} - \text{number of $\mathsf{1}$'s among the $e_{i}'$'s}; \\
y = \text{number of $\mathsf{2}$'s among the $e_{i}$'s} - \text{number of $\mathsf{2}$'s among the $e_{i}'$'s.}
\end{align*}
As a corollary, 
\begin{align*}
\text{number of $\mathsf{0}$'s among the $e_{i}'$'s} - \text{number of $\mathsf{0}$'s among the $e_{i}$'s} = x+y \geq 0.
\end{align*}
\begin{figure}[h]
\includegraphics[width = 
0.9\textwidth]{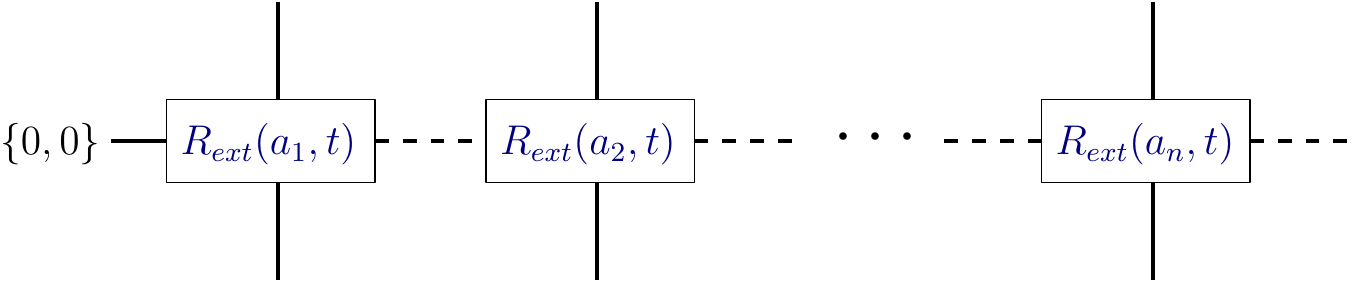}
\caption{$H$ is a transfer operator for the inhomogeneous vertex model defined via $R_{ext}$. Boundary condition is fixed to be $\{0, 0\}$ on the left and is free on the right.} 
\label{ExtendedTransferMatrix}
\end{figure}

For $0 \leq k_{1}, k_{2} \leq n$ define an operator $H_{k_{1}, k_{2}}:W_{n, k_{1}} \to W_{n, k_{2}}$ via restriction of $H$. Then $H_{k_{1}, k_{2}} = 0$ unless $k_{1} \leq k_{2}$. Then theorem \ref{HLpositivity} would follow from the following
\begin{theorem}
\label{ExtendedVertexTheorem}
Let $\lambda$ be a partition with $\lambda_{1} = m$. Then 
\begin{align}
\label{ExtendedVertexRepresentation}
\Theta(P_{\lambda}) = \pi^{\otimes n} \circ H_{\lambda_{2}', \lambda_{1}'} \circ H_{\lambda_{3}', \lambda_{2}'} \circ \cdots \circ H_{\lambda_{m}', \lambda_{m-1}'} \circ H_{0, \lambda_{m}'} \circ i^{\otimes n}
\end{align}
\end{theorem}
\begin{remark}
Note that although $\Theta(P_{\lambda})$ is itself an operator $V^{\otimes n} \to V^{\otimes n}$, the right hand side of representation \eqref{ExtendedVertexRepresentation} utilizes a bigger space $W^{\otimes n}$ in its intermediate steps.
\end{remark}

\section{Proof of the main result}
\label{Proof}
\begin{proof}
Hall-Littlewood polynomials (as well as more general Macdonald polynomials) satisfy the Pieri formulas (see \cite{Mac99}, pp.340-341):
\begin{align}
\label{PieriFormulas}
P_{\lambda}g_{r} = \sum_{\lambda \prech \mu, \ |\mu|-|\lambda|=r}\phi_{\mu/\lambda}(0, t)P_{\mu}, \qquad
P_{\lambda}e_{r} = \sum_{\lambda \precv \mu, \ |\mu|-|\lambda|=r}\psi'_{\mu/\lambda}(0, t)P_{\mu}. 
\end{align}
We will adopt the conventions that $\lambda_{m}' = 0$ for $m > \lambda_{1}$, $\mu_{m}'=0$ for $m > \mu_{1}$, and $\lambda_{0}' = \mu_{0}'$. The multiplicities $\phi_{\mu/\lambda} = \phi_{\mu/\lambda}(0, t)$ and $\psi'_{\mu/\lambda} = \psi'_{\mu/\lambda}(0, t)$ in \eqref{PieriFormulas} can be expressed as
\begin{flalign}
\label{phirep}
& \phi_{\mu/\lambda} = \prod_{i=1}^{\mu_{1}} \mathbf{1}_{\mu_{i}'=\lambda_{i}'+1, \ \mu'_{i+1}=\lambda_{i+1}'}\cdot \left(1 -t^{\mu_{i}'-\lambda_{i+1}'}\right)  = \nonumber \\ 
& =
\prod_{i=1}^{\mu_{1}}
\mathbf{1}_{\mu_{i}'=\lambda_{i}'+1} \cdot \left( \mathbf{1}_{\mu_{i-1}'=\lambda_{i-1}'} \cdot \left(1 -t^{\mu_{i}'-\lambda_{i+1}'}\right) + \mathbf{1}_{\mu_{i-1}'=\lambda_{i-1}'+1} \cdot \left(\frac{1 -t^{\mu_{i}'-\lambda_{i+1}'}}{1-t^{\mu_{i-1}'-\lambda_{i}'}}\right) \right) = \\ 
& = \left(1 - t \cdot \mathbf{1}_{\mu_{1} > \lambda_{1}} \right) \cdot \prod_{i=1}^{\lambda_{1}}
\mathbf{1}_{\mu_{i}'=\lambda_{i}'+1} \cdot \left( \mathbf{1}_{\mu_{i-1}'=\lambda_{i-1}'} \cdot \left(1 -t^{\mu_{i}'-\lambda_{i+1}'}\right) + \mathbf{1}_{\mu_{i-1}'=\lambda_{i-1}'+1} \cdot \left(\frac{1 -t^{\mu_{i}'-\lambda_{i+1}'}}{1-t^{\mu_{i-1}'-\lambda_{i}'}}\right) \right) \nonumber
\end{flalign}
\begin{flalign}
\label{psiprimerep}
\psi'_{\mu/\lambda} = \prod_{i=1}^{\lambda_{1}} \frac{(t;t)_{\mu'_{i}-\mu'_{i+1}}}{(t;t)_{\mu'_{i}-\lambda'_{i}}(t;t)_{\lambda'_{i}-\mu'_{i+1}}} = \prod_{i=1}^{\mu_{1}} \left(\frac{(t;t)_{\mu'_{i}-\lambda'_{i+1}}}{(t;t)_{\mu'_{i}-\lambda'_{i}}(t;t)_{\lambda'_{i}-\lambda'_{i+1}}} \cdot \frac{(t;t)_{\mu'_{i-1}-\mu'_{i}}(t;t)_{\lambda'_{i-1}-\lambda'_{i}}}{(t;t)_{\mu'_{i-1}-\lambda'_{i}}(t;t)_{\lambda'_{i-1}-\mu'_{i}}}  \right)
\end{flalign}
Denote by $\Pi_{\lambda}$ the right hand side of \eqref{ExtendedVertexRepresentation}. To prove theorem \ref{ExtendedVertexTheorem} we need to show that $\Pi(P_{\lambda}) = \Pi_{\lambda}$ for any partition $\lambda$. 
It is enough to prove that for any partition $\lambda$ we have 
\begin{align}
\label{HorizontalPieriRelation}
T(\alpha)\Pi_{\lambda} = \left( \prod_{i=1}^{n} \frac{1 - \alpha a_{i}}{1 - t \alpha a_{i}} \right)\sum_{\lambda \prech \mu} \alpha^{|\mu|-|\lambda|}\phi_{\mu/\lambda} \Pi_{\mu}.
\end{align}
Indeed, suppose we know that relation \ref{HorizontalPieriRelation} holds. Multiply both sides of this relation by $\displaystyle \prod_{i=1}^{n} \frac{1-t\alpha a_{i}}{1 - \alpha a_{i}}$,  then take coefficient of $\alpha^r$ to get
\begin{align}
\label{HP}
T_{r}\Pi_{\lambda} = \sum_{\lambda \prech \mu, \ |\mu|-|\lambda|=r} \phi_{\mu/\lambda} \Pi_{\mu} \qquad \text{for any $r \in \mathbb{Z}_{\geq 0}$}.
\end{align}
On the other hand, it follows from the definition of $\Theta(P_{\lambda})$ that 
\begin{align}
T_{r}\Theta(P_{\lambda}) = \sum_{\lambda \prech \mu, \ |\mu|-|\lambda|=r} \phi_{\mu/\lambda} \Theta(P_{\mu}) \qquad \text{for any $r \in \mathbb{Z}_{\geq 0}$}.
\end{align}
We can now show that $\Theta(P_{\lambda}) = \Pi_{\lambda}$ by induction on $(\lambda_{1}', \lambda_{-1})$. Base for $(0, 0)$ is clear: $H_{k, k}(v) = v$ for any $v \in \langle \mathsf{1}, \mathsf{2} \rangle ^{\otimes n}$, hence $\Pi_{\emptyset} = Id = \Theta(P_{\emptyset})$. Suppose the equality $\Theta(P_{\lambda}) = \Pi_{\lambda}$ has been established for all $\lambda$ with either $\lambda_{1}' < \nu_{1}'$ or  both $\lambda_{1}' = \nu_{1}'$ and $\lambda_{-1} < \nu_{-1}$. We would like to also establish it for $\nu$. Let $\chi$ be the partition obtained from $\nu$ by deleting its last nonzero row. Then by \eqref{HP} and the inductive assumption we get 
\begin{multline*}
\Theta(P_{\nu}) = \frac{1}{\phi_{\nu/\chi}}\left(T_{\nu_{-1}}\Theta(P_{\chi}) - \sum_{\chi \prech \lambda, \ |\lambda|-|\chi|=\nu_{-1}, \ \lambda \neq \nu} \phi_{\lambda/\chi} \Theta(P_{\lambda}) \right) = \\ = \frac{1}{\phi_{\nu/\chi}}\left( T_{\nu_{-1}}\Pi_{\chi} - \sum_{\chi \prech \lambda, \ |\lambda|-|\chi|=\nu_{-1}, \ \lambda \neq \nu} \phi_{\lambda/\chi} \Pi_{\lambda} \right) = \Pi_{\nu}.
\end{multline*}
So it remains to establish \eqref{HorizontalPieriRelation}. Consider a stochastic vertex model $R_{3}(a, t)$ with weights as depicted on Fig. \ref{StochasticTripleVertex}.  
\begin{figure}[h]
\includegraphics[width = 
0.8\textwidth]{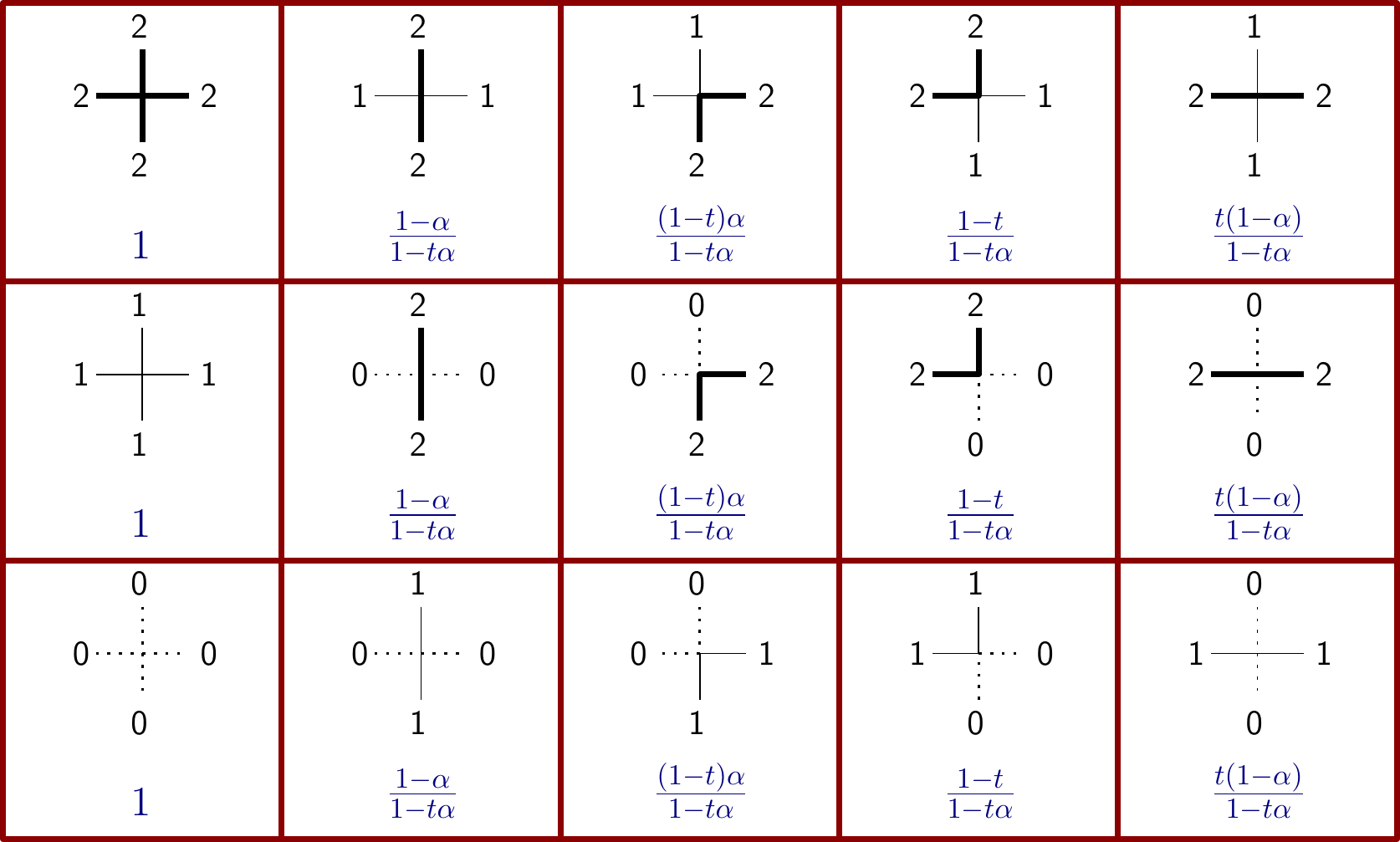}
\caption{Stochastic vertex model with weights given by $R_{3}(a, t)$. Fat lines correspond to $\mathsf{2}$'s, normal lines correspond to $\mathsf{1}$'s, dotted lines correspond to $\mathsf{0}$'s.}
\label{StochasticTripleVertex}
\end{figure}

Denote by $\mathcal{A}_{k}$  the inhomogeneous transfer operator $W_{n, k} \to W_{n, k}$ of $R_3$ with both left and right boundary conditions fixed to be $\mathsf{0}$ (as specified on top of Fig. \ref{TransferMatrixAB}). Denote by $\mathcal{B}_{k}$  the inhomogeneous transfer operator $W_{n, k} \to W_{n, k+1}$ of $R_3$ with the left boundary condition fixed to be $\mathsf{0}$ and the right boundary condition being $\mathsf{1}$ or $\mathsf{2}$ (as specified on bottom Fig. \ref{TransferMatrixAB}). 

\begin{figure}[h]
\includegraphics[width = 
0.9\textwidth]{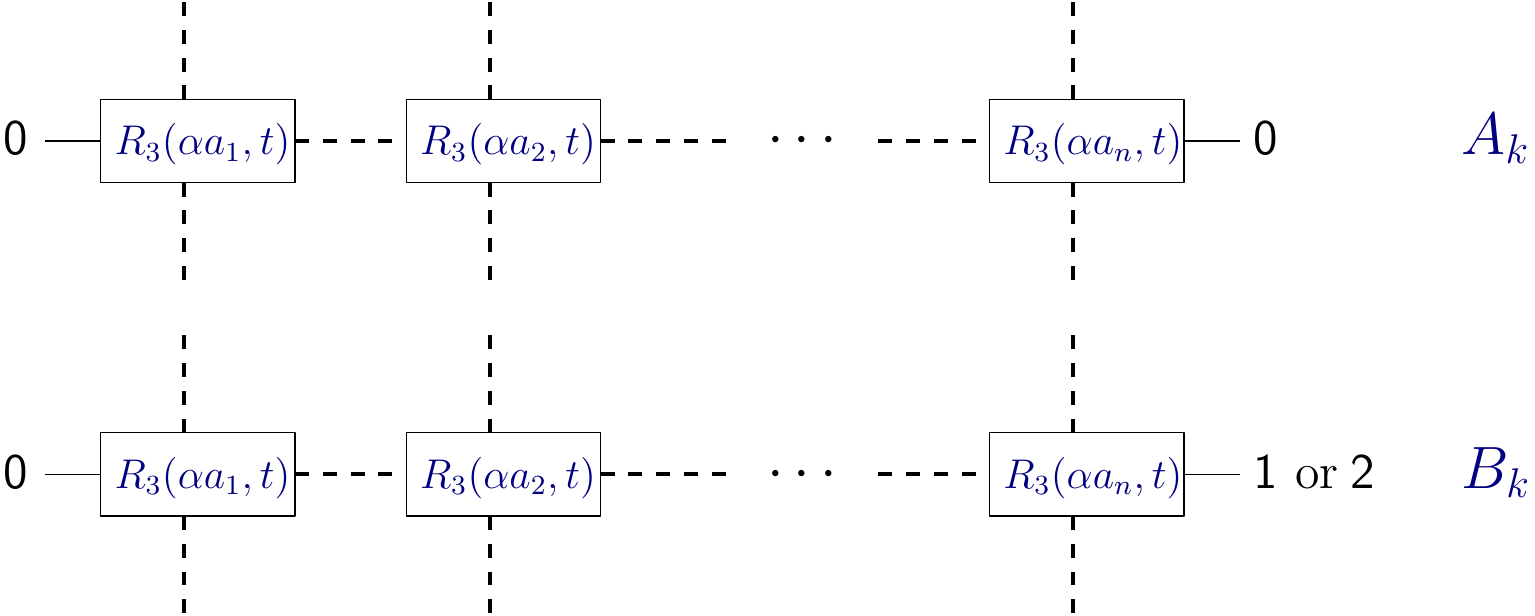}
\caption{Top: Transfer operator $\mathcal{A}_k$. Bottom: Transfer operator $\mathcal{B}_k$.}
\label{TransferMatrixAB}
\end{figure}

\begin{lemma}
\label{KeyHrelationsLemma}
For $0 \leq k_{1} \leq k_{2} \leq n$ we have
\begin{flalign}
\label{Acommutation}
& \mathcal{A}_{k_{2}}H_{k_{1}, k_{2}} = H_{k_{1}, k_{2}}\mathcal{A}_{k_{1}} + H_{k_{1}+1, k_{2}}\mathcal{B}_{k_{1}}; \\
\label{Bcommutation}
& \mathcal{B}_{k_{2}}H_{k_{1}, k_{2}} = \alpha\left(1-t^{k_{2}-k_{1}+1}\right)H_{k_{1}, k_{2}+1}\mathcal{A}_{k_{1}} + \alpha H_{k_{1}+1, k_{2}+1}\mathcal{B}_{k_{1}}.
\end{flalign}
\end{lemma}
\begin{proof}[Proof of lemma \ref{KeyHrelationsLemma}] To prove relations \eqref{Acommutation} and \eqref{Bcommutation} we will utilize a Yang-Baxter type equation relating $R_{3}$, $R_{ext}$ and another (auxiliary) vertex model $R_{aux}(t): U_{2} \otimes \langle \mathsf{0}, \mathsf{1}, \mathsf{2} \rangle \to U_{2} \otimes \langle \mathsf{0}, \mathsf{1}, \mathsf{2} \rangle$ with weights as specified on Fig. \ref{ExtraR}.
\begin{figure}[h]
\includegraphics[width = 
0.7\textwidth]{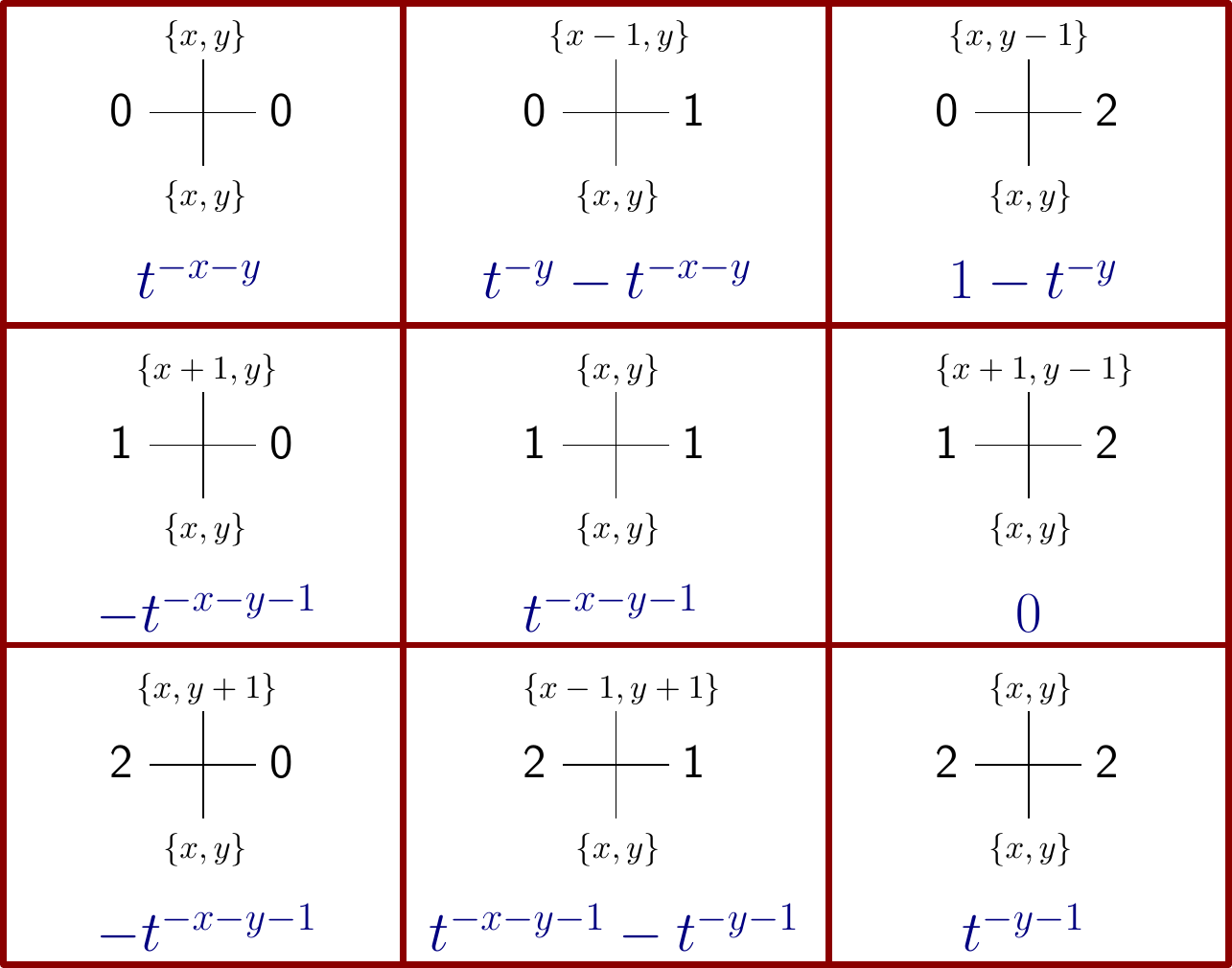}
\caption{Weights of the vertex model defined via $R_{aux}(t)$. All configurations not appearing on this picture are assumed to have weight $0$.}
\label{ExtraR}
\end{figure}
The Yang-Baxter type equation we need is an equality of operators $\langle \mathsf{0}, \mathsf{1}, \mathsf{2} \rangle \otimes U_{2} \otimes \langle \mathsf{0}, \mathsf{1}, \mathsf{2} \rangle \to \langle \mathsf{0}, \mathsf{1}, \mathsf{2} \rangle \otimes U_{2} \otimes \langle \mathsf{0}, \mathsf{1}, \mathsf{2} \rangle$ as specified on Fig. \ref{HPYangBaxter}.
\begin{figure}[h]
\includegraphics[width = 
0.9\textwidth]{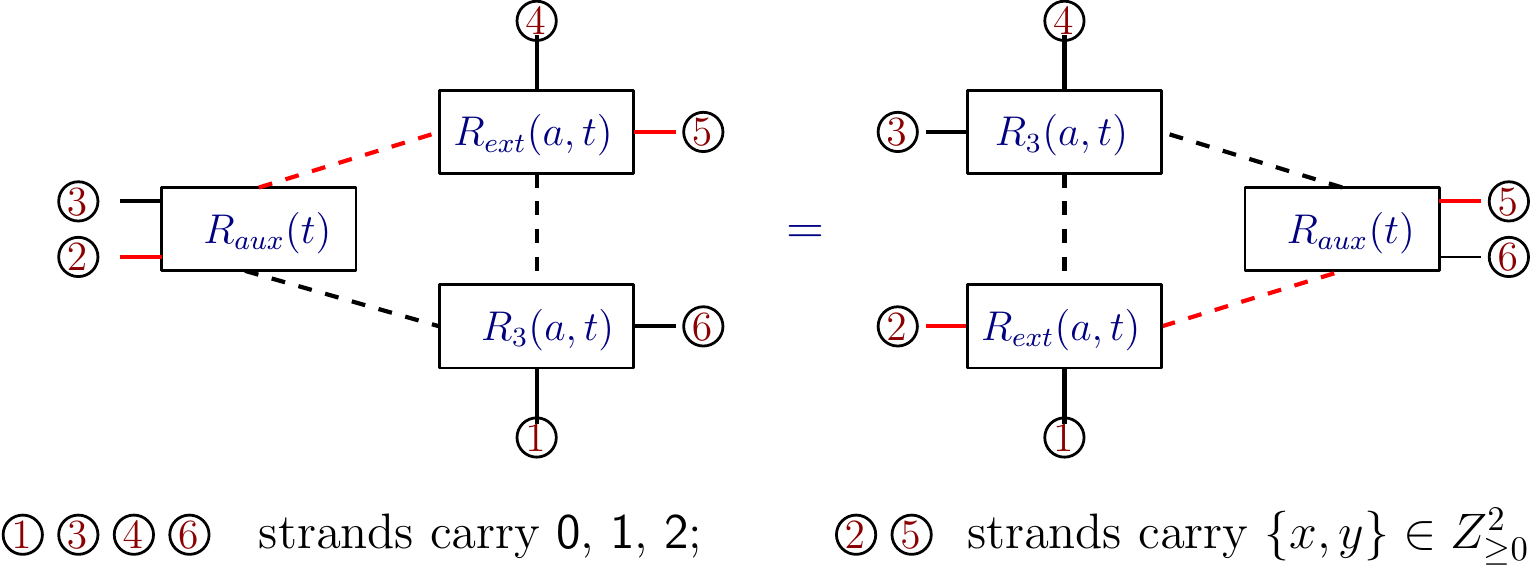}
\caption{Yang-Baxter type equation for operators $R_{3}$, $R_{ext}$, $R_{aux}$.}
\label{HPYangBaxter}
\end{figure}

Note that $R_{aux}((0, 0) \otimes \mathsf{0}) = (0, 0) \otimes \mathsf{0}$ and for any $(x, y) \in \mathbb{Z}_{\geq 0}^{2}$ we have 
\begin{align}
\label{Rauxrel1}
R_{aux}((x, y), \mathsf{0}, (x, y), \mathsf{0})+ R_{aux}((x, y), \mathsf{0}, (x-1, y), \mathsf{1})+R_{aux}((x, y), \mathsf{0}, (x, y-1), \mathsf{2}) = 1, \\
R_{aux}((x-1, y), \mathsf{1}, (x, y), \mathsf{0})+ R_{aux}((x-1, y), \mathsf{1}, (x-1, y), \mathsf{1})+R_{aux}((x-1, y), \mathsf{1}, (x, y-1), \mathsf{2}) = 0, \nonumber \\
R_{aux}((x, y-1), \mathsf{2}, (x, y), \mathsf{0})+ R_{aux}((x, y-1), \mathsf{2}, (x-1, y), \mathsf{1})+R_{aux}((x, y-1), \mathsf{2}, (x, y-1), \mathsf{2}) = 0. \nonumber
\end{align}
Multiply both sides of \eqref{Acommutation} by $\alpha^{k_{2}}$. Then this equality becomes a corollary of a repeated application of the Yang-Baxter type equation from Fig. \ref{HPYangBaxter} together with \eqref{Rauxrel1}. More precisely, see the chain of equalities \eqref{Acommutationchain}
 
\begin{multline}
\includegraphics[scale=1]{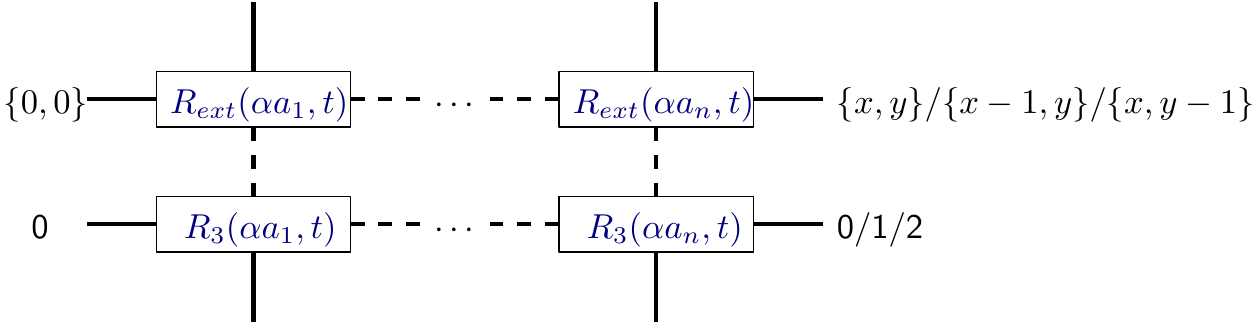} \scalebox{2}{$=$} \\ \includegraphics[scale=1]{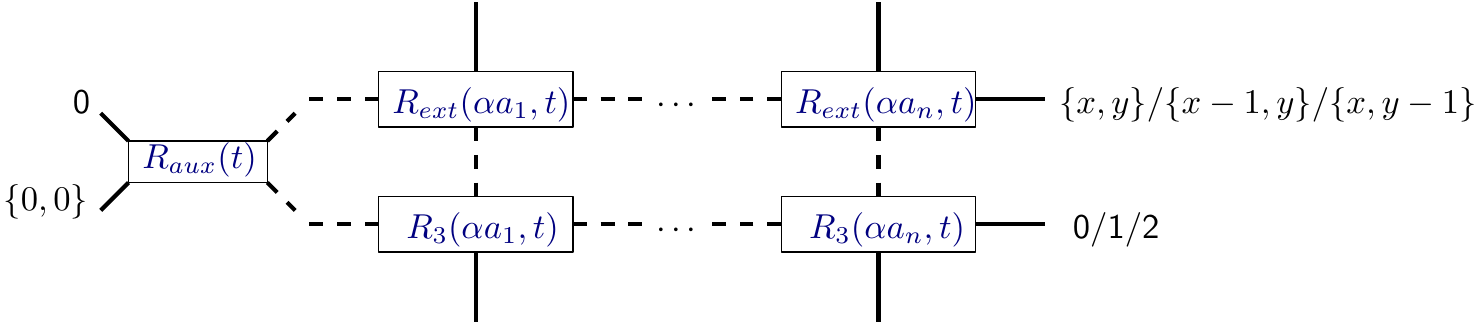} \scalebox{2}{$=$} \\
\includegraphics[scale=1]{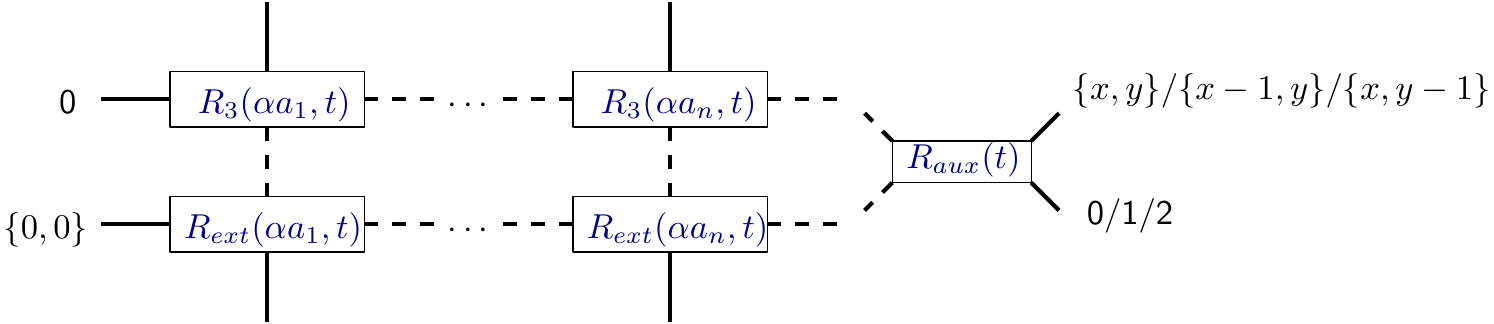} \scalebox{2}{$=$} \\
\includegraphics[scale=1]{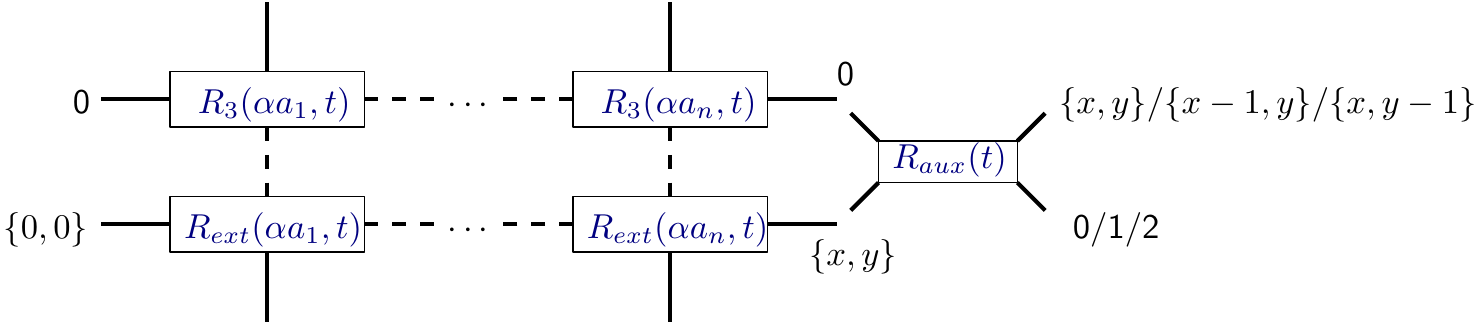} \scalebox{2}{$=$} \\
\includegraphics[scale=1]{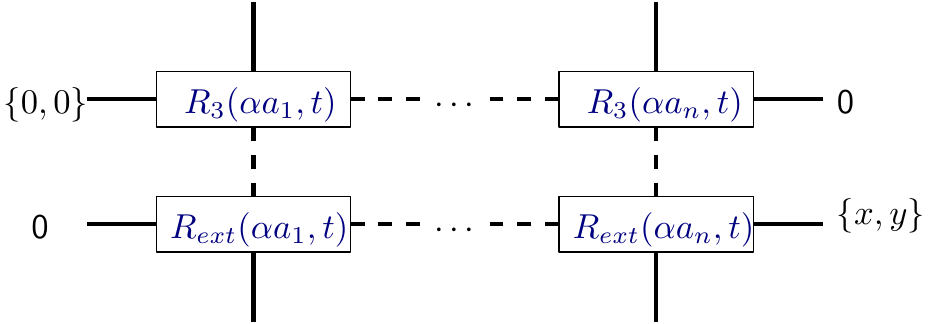}
\label{Acommutationchain}
\end{multline}

Note that for any $(x, y) \in \mathbb{Z}_{\geq 0}^{2}$ we have 
\begin{align}
\label{Rauxrel2}
 R_{aux}((x, y), \mathsf{0}, (x, y), \mathsf{0})\cdot (1-t^{x+y})+ R_{aux}((x, y), \mathsf{0}, (x-1, y), \mathsf{1})+ R_{aux}((x, y), \mathsf{0}, (x, y-1), \mathsf{2}) = 0, \nonumber \\
 R_{aux}((x-1, y), \mathsf{1}, (x, y), \mathsf{0}) \cdot (1-t^{x+y}) + R_{aux}((x-1, y), \mathsf{1}, (x-1, y), \mathsf{1})+ \nonumber \\ + R_{aux}((x-1, y), \mathsf{1}, (x, y-1), \mathsf{2}) = 1, \\
R_{aux}((x, y-1), \mathsf{2}, (x, y), \mathsf{0})\cdot (1-t^{x+y})+ R_{aux}((x, y-1), \mathsf{2}, (x-1, y), \mathsf{1}) + \nonumber \\ + R_{aux}((x, y-1), \mathsf{2}, (x, y-1), \mathsf{2}) = 1. \nonumber
\end{align}

Multiply both sides of \eqref{Bcommutation} by $\alpha^{k_{2}}$. Then this equality becomes a corollary of a repeated application of the Yang-Baxter type equation from Fig. \ref{HPYangBaxter} together with \eqref{Rauxrel2}. More precisely, see the chain of equalities \eqref{Bcommutationchain}

\begin{multline}
\includegraphics[scale=1]{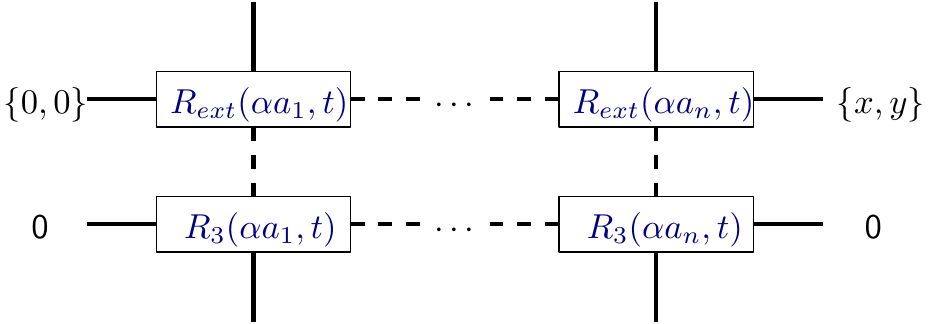} \scalebox{2}{$+$} \\ \includegraphics[scale=1]{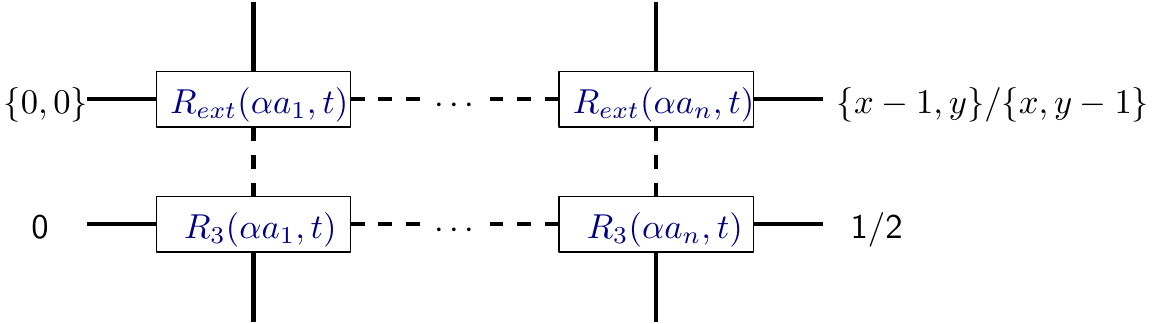} \scalebox{2}{$=$} \\ 
 \includegraphics[scale=1]{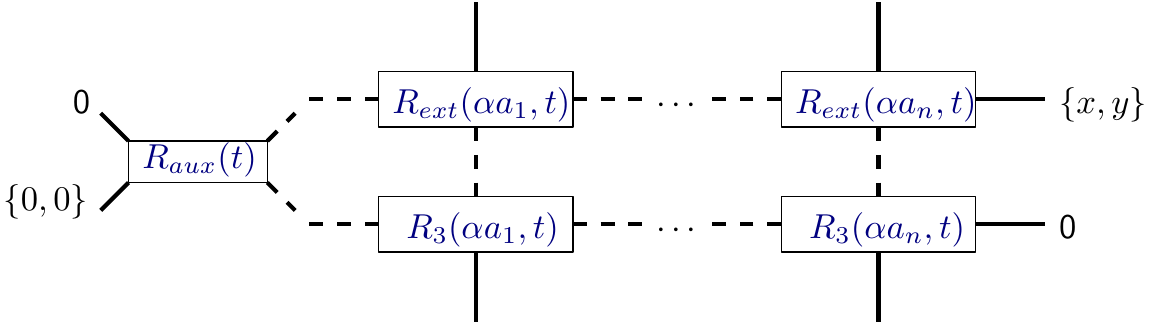} \scalebox{2}{$+$} \\ \includegraphics[scale=1]{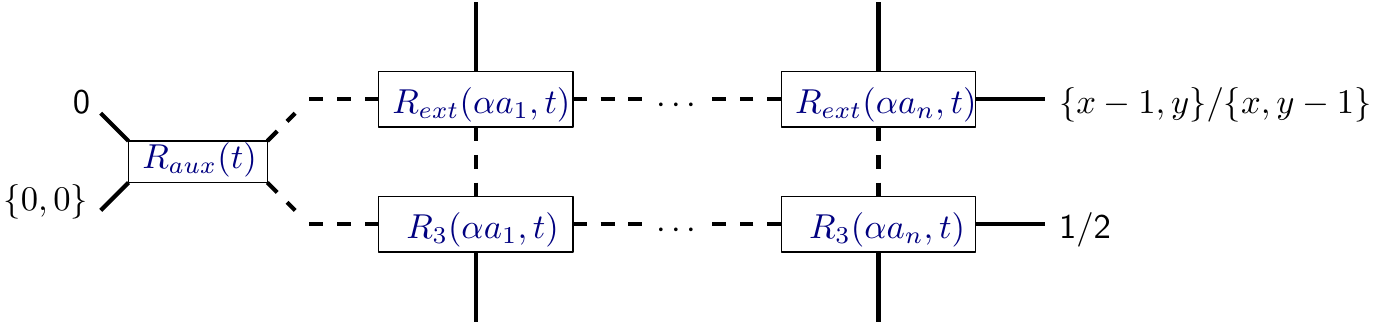} \scalebox{2}{$=$} \\ 
\includegraphics[scale=1]{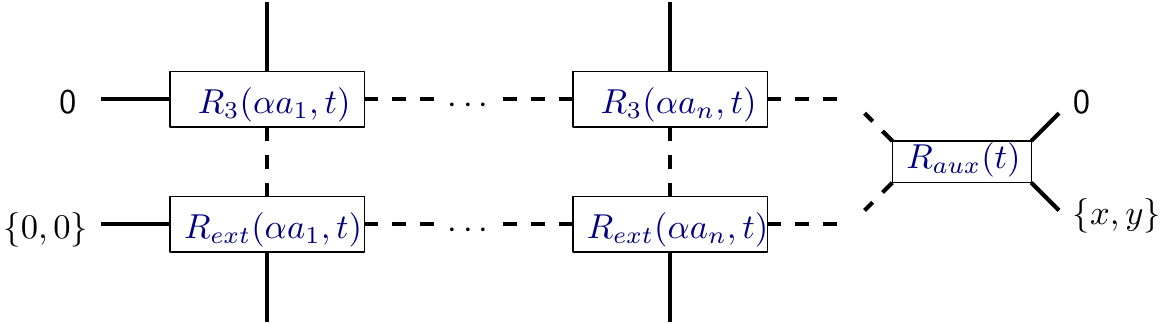} \scalebox{2}{$+$} \\ \includegraphics[scale=1]{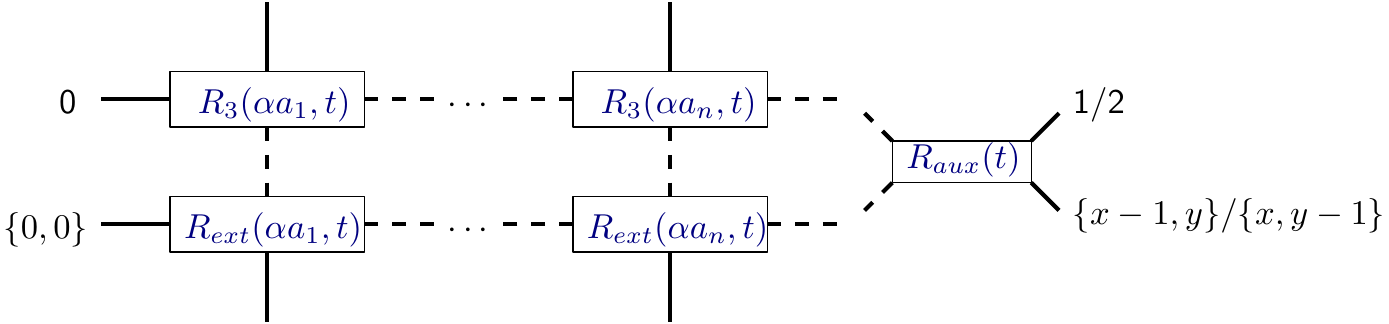} \scalebox{2}{$=$} \\
\includegraphics[scale=1]{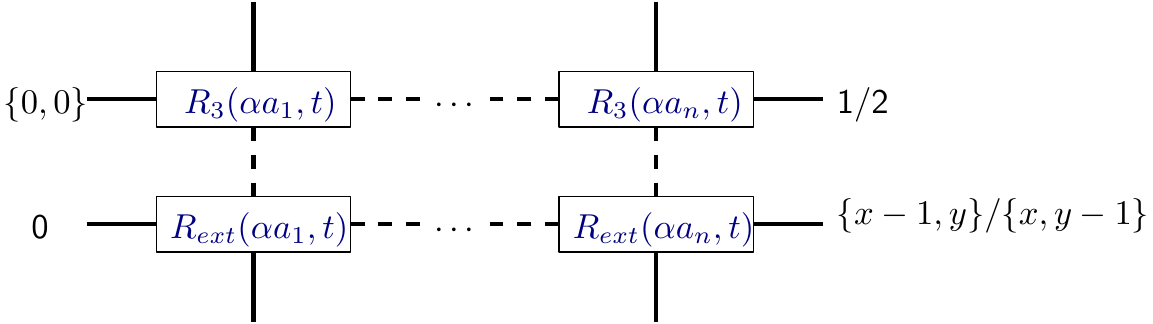} 
\label{Bcommutationchain}
\end{multline}
\end{proof}
It remains to show that \eqref{HorizontalPieriRelation} can be proved for any given $\lambda$ using Lemma \ref{KeyHrelationsLemma}.  Consider the operator  $\left(\mathcal{A}_{\lambda_{1}'} + \mathcal{B}_{\lambda_{1}'}  \right)  \circ H_{\lambda'_{2}, \lambda'_{1}} \circ   \cdots \circ H_{\lambda'_{-1}, \lambda'_{-2}} \circ  H_{0, \lambda'_{-1}}$ and repeatedly use the commutation relations \eqref{Acommutation} and \eqref{Bcommutation} to move $\mathcal{A}$'s and $\mathcal{B}$'s through all the $H$'s. Note that  \eqref{Acommutation} for $k_{1} = k_{2} =k$ becomes $\mathcal{A}_{k}H_{k, k} = H_{k, k}\mathcal{A}_{k}$. Then the resulting sum is
\begin{multline}
\sum_{\mu: \ \lambda \prech \mu, \ \mu_{1} = \lambda_{1}} \left( \left[\alpha^{|\mu| - |\lambda|} \prod_{1 \leq i \leq \lambda_{1}: \ \mu_{i}'=\lambda_{i}'+1, \ \mu'_{i+1}=\lambda_{i+1}'} \left(1 -t^{\mu_{i}'-\lambda_{i+1}'}\right) \right] H_{\mu'_{2}, \mu'_{1}} \circ   \cdots \circ H_{\mu'_{-1}, \mu'_{-2}} \circ  H_{0, \mu'_{-1}} \circ \mathcal{A}_{0} \right. \\
 \left. + \left[\alpha^{|\mu| - |\lambda|} \prod_{1 \leq i \leq \lambda_{1}-1: \ \mu_{i}'=\lambda_{i}'+1, \ \mu'_{i+1}=\lambda_{i+1}'} \left(1 -t^{\mu_{i}'-\lambda_{i+1}'}\right) \right] H_{\mu'_{2}, \mu'_{1}} \circ   \cdots \circ H_{\mu'_{-1}, \mu'_{-2}} \circ  H_{1, \mu'_{-1}} \circ \mathcal{B}_{0} \right).
 \label{AfterHcommutation}
\end{multline}
Note that $\displaystyle \mathcal{A}_{0} =  \left( \prod_{i=1}^{n} \frac{1 - \alpha a_{i}}{1 - t \alpha a_{i}} \right) Id$. Substitute $k_{1} = k_{2} = 0$ in \eqref{Bcommutation} to get
\begin{align*}
\mathcal{B}_{0}H_{0, 0} = \alpha(1-t)H_{0, 1}\mathcal{A}_{0} + \alpha H_{1, 1}\mathcal{B}_{0}.
\end{align*}
$H_{0, 0} = Id$, so 
\begin{align*}
\mathcal{B}_{0} = \alpha(1-t) \left( \prod_{i=1}^{n} \frac{1 - \alpha a_{i}}{1 - t \alpha a_{i}} \right)  \left(Id - \alpha H_{1, 1} \right)^{-1}H_{0, 1} = (1-t) \left( \prod_{i=1}^{n} \frac{1 - \alpha a_{i}}{1 - t \alpha a_{i}} \right)  \sum_{k=1}^{\infty} \alpha^{k}H_{1, 1}^{k-1}H_{0, 1}.
\end{align*}
Substituting this equality in \eqref{AfterHcommutation} gives us
\begin{multline}
\label{HorizontalPieriRelation2}
\left(\mathcal{A}_{\lambda_{1}'} + \mathcal{B}_{\lambda_{1}'}  \right)  H_{\lambda'_{1}, \lambda'_{2}} \circ  \cdots \circ H_{\lambda'_{-2}, \lambda'_{-1}} \circ  H_{\lambda'_{-1}, 0}   \\ = \left( \prod_{i=1}^{n} \frac{1 - \alpha a_{i}}{1 - t \alpha a_{i}} \right)\sum_{\lambda \prech \mu} \alpha^{|\mu|-|\lambda|}\phi_{\mu/\lambda} H_{\mu'_{2}, \mu'_{1}} \circ   \cdots \circ H_{\mu'_{-1}, \mu'_{-2}} \circ  H_{0, \mu'_{-1}}. 
\end{multline}
Finally, note that $T(\alpha) \pi^{\otimes n} = \pi^{\otimes n} \left(\mathcal{A}_{\lambda_{1}'} + \mathcal{B}_{\lambda_{1}'}  \right)$ as operators $W_{n, \lambda_{1}'} \to W_{n, 0}$. Together with \eqref{HorizontalPieriRelation2} this implies \eqref{HorizontalPieriRelation}.

\end{proof}

\end{document}